\documentclass{amsart}

\pdfoutput=1

\usepackage[utf8]{inputenc}
\usepackage{bm}
\usepackage{diagbox}
\usepackage[colorlinks, citecolor=red, breaklinks=true]{hyperref} 
\usepackage{xcolor}
\usepackage{graphicx}
\usepackage{mathtools}
\usepackage{todonotes}
\usepackage{commath}
\usepackage[square, numbers, sort]{natbib}

\usepackage{paralist,enumitem}
\setlist{listparindent=0pt,parsep=3pt}

\newtheorem{theorem}{Theorem}[section]
\newtheorem{lemma}[theorem]{Lemma}

\newtheorem{proposition}[theorem]{Proposition}
\newtheorem{fact}[theorem]{Fact}

\theoremstyle{definition}
\newtheorem{definition}[theorem]{Definition}

\theoremstyle{remark}
\newtheorem{remark}[theorem]{Remark}
\newtheorem{example}[theorem]{Example}

\numberwithin{equation}{section}

\newcommand{\hthree}{\mathbb{H}^3}
\newcommand{\dsthree}{\mathbb{S}^3_1}

\newcommand{\sltc}{\mathrm{SL}(2,\mathbb{C})}
\newcommand{\hermtwo}{\mathrm{Herm}(2)}

\newcommand{\qthreep}{\mathbb{Q}^3_+}
\newcommand{\qthreem}{\mathbb{Q}^3_-}
\newcommand{\lfour}{\mathbb{L}^4}
\newcommand{\ithree}{\mathbb{I}^3}

\newcommand{\sltr}{\mathrm{SL}(2,\mathbb{R})}
\newcommand{\lasltr}{\mathfrak{sl}(2,\mathbb{R})}
\newcommand{\lasltc}{\mathfrak{sl}(2,\mathbb{C})}

\newcommand{\mbbH}{\mathbb{H}}
\newcommand{\mbbS}{\mathbb{S}}

\newcommand{\rmx}{\mathrm{x}}
\newcommand{\rmy}{\mathrm{y}}

\newcommand{\rmxt}{\mathrm{x}_0}
\newcommand{\rmxx}{\mathrm{x}_1}
\newcommand{\rmxy}{\mathrm{x}_2}
\newcommand{\rmxz}{\mathrm{x}_3}

\newcommand{\T}{\mathbf{T}}
\newcommand{\N}{\mathbf{N}}
\newcommand{\B}{\mathbf{B}}
\newcommand{\G}{\mathrm{G}}
\newcommand{\gammab}{\bm{\gamma}}

\newcommand{\f}{\bm{f}}

\newcommand{\tr}[1]{\operatorname{tr} #1}

\newcommand{\mcU}{\mathcal{U}}

\newenvironment{contentsinred}{\bgroup\color{red}}{\egroup}
\newcommand{\br}{\begin{contentsinred}}
\newcommand{\er}{\end{contentsinred}}

\begin{document}

\title[Ruled ZMC surfaces in $\qthreep$]{Ruled zero mean curvature surfaces \\ in the three-dimensional light cone}

\author[J.\ Cho]{Joseph Cho}
\address[Joseph Cho]{
Global Leadership School, Handong Global University, 558 Handong-ro Buk-gu, Pohang, Gyeongsangbuk-do 37554, Republic of Korea}
\email{jcho@handong.edu}

\author[D.\ Lee]{Dami Lee}
\address[Dami Lee]{
Department of Mathematics, Oklahoma State University, Stillwater, OK 74078, U.S.A.}
\email{dami.lee@okstate.edu}

\author[W.\ Lee]{Wonjoo Lee}
\address[Wonjoo Lee]{
Department of Mathematics, Korea University, Seoul 02841, Republic of Korea}
\email{wontail123@korea.ac.kr}

\author[S.-D.\ Yang]{Seong-Deog Yang}
\address[Seong-Deog Yang]{
Department of Mathematics, Korea University, Seoul 02841, Republic of Korea}
\email{sdyang@korea.ac.kr}


\begin{abstract}
	We obtain a complete classification of ruled zero mean curvature surfaces in the three-dimensional lightcone.
	En route, we examine geodesics and screw motions in the space form, allowing us to discover helicoids.
	We also consider their relationship to catenoids using Weierstrass representations of zero mean curvature surfaces in the three-dimensional lightcone.
\end{abstract}

\maketitle

\section{Introduction}\label{Section1}
Rooted in the discoveries of non-Euclidean geometry, it is often interesting to examine which Euclidean geometric concepts and results can be applied to other geometries.
Ever since Catalan showed that the standard helicoid is the only non-trivial ruled minimal surface in Euclidean space \cite{catalan_SurfacesRegleesDont_1842}, classification problem of ruled surfaces satisfying certain curvature restrictions in various space forms has received plethora of interest across various three-dimensional spaces.
These results include the study of:
\begin{itemize}
	\item ruled Weingarten surfaces in Lorentzian $3$-space \cite{DK1},
	\item ruled and helicoidal zero mean curvature surfaces in Lorentzian $3$-space $\mathbb{L}^3$ \cite{MP1},
	\item ruled zero mean curvature (ZMC) surfaces in $\mbbS^2 \times \mathbb{R}$, $\mbbH^2 \times \mathbb{R}$, Heisenberg group $\mathrm{Nil}_3$, and the Berger sphere \cite{kim_HelicoidsBbbS^2timesBbb_2009, shin_RuledMinimalSurfaces_2013, shin_RuledMinimalSurfaces_2015}, and
	\item ruled Weingarten hypersurfaces in hyperbolic spaces $\mbbH^{n+1}$, Lorentzian spaces and de Sitter spaces $\mbbS^{n+1}$ \cite{AV1,ACV1, ALV1}, and
	\item ruled ZMC surfaces in isotropic $3$-space \cite{Vogel}.
\end{itemize}


The main purpose of this article is to present a classification of ruled ZMC surfaces in the three-dimensional light cone $\qthreep$, which is a space form equipped with a degenerate metric.
Its curve theory and surface theory have been developed in \cite{Liu, Liu2, Liu3, Liu4}, and we review and organize these in our preparatory Section~\ref{Section2}.
We also review rotational ZMC surfaces, called \emph{catenoids}, in $\qthreep$, found in \cite{CKLLY1}, as they will play an important role.

We begin our main discussion in Section~\ref{Section3} by considering \emph{geodesics} in $\qthreep$, as degenerate metric of the space form present a meaningful obstacle.
Then we find two important ruled ZMC surfaces in Section~\ref{Subsec:202503230110PM}, to serve as examples unveiling the required ansatz to obtain a classification of ruled ZMC surfaces in $\qthreep$.
In particular, noting that helicoids in Euclidean space can be characterized as ruled surfaces that are invariant under screw motions, we consider screw motions in Section~\ref{sect:screwMotion}, and show that the surface obtained via applying screw motions to a geodesic has ZMC in Section~\ref{SubSec:202503231257PM}, which we call \emph{helicoids}.
We also show that one of the catenoids reviewed, parabolic catenoid, is in fact a ruled surface in Section~\ref{SubSec:202503231258PM}.

A noted feature of ZMC surfaces in $\qthreep$ is that they admit a \emph{Weierstrass representation}\footnote{On contrary, ZMC surfaces in hyperbolic space or de Sitter space has a Weierstrass-type representation that is different in flavor, exploiting the method of infinite-dimensional Lie groups, known as loop groups.} \cite{Liu4, Liu2, seo_ZeroMeanCurvature_2021} as in the case of minimal surfaces in Euclidean space \cite{weierstrass_UntersuchungenUberFlachen_1866}.
Weierstrass representation allows for a zero mean curvature surface to be represented in terms of a meromorphic function and a holomorphic $1$-form called the \emph{Weierstrass data}.
Notably, the famous isometric deformation connecting catenoids to helicoids in Euclidean space admits a simple characterization in terms of its Weierstrass data, known as the associated family.
Thus, in Section~\ref{Section4}, we present a comprehensive examination of the relationship between catenoids and helicoids in $\qthreep$.
In particular, we show that, as in the Euclidean case, every helicoid is in the associated family of some catenoid (Theorem~\ref{thm:helicate}), but unlike the Euclidean case, only certain catenoids admit a helicoid in its associated family (Theorem \ref{thm:cateheli}).
Then, in Section~\ref{sec:lawson}, we introduce a Lawson-type correspondence between ZMC surfaces in the isotropic $3$-space and $\qthreep$ and compare the associated families of catenoid in each space form under the correspondence.

Finally, in Section~\ref{Section5}, we obtain a complete classification of ruled ZMC surfaces in $\qthreep$ (Theorem~\ref{Thm:202503130757AM}), showing that every ruled ZMC surface must be a helicoid or a parabolic catenoid up to isometries and homotheties of $\qthreep$.


\textbf{Acknowledgements.} The authors would like to thank Prof.\ Heayong Shin for his constant encouragements on this work.

\section{Preliminaries}\label{Section2}
In this section, we briefly review the basic differential geometry of three dimensional lightcone, including curve and surface theory, mainly to set the notations to be used throughout. (We refer the readers to \cite{Liu, Liu4} for a detailed introduction.)

\subsection{$\qthreep$ as a quadric}
We identify Lorentzian $4$-space $\mathbb{L}^4$ with $\hermtwo$ as follows:
	\[
		\lfour \ni (\rmxt, \rmxx, \rmxy, \rmxz) \sim \begin{pmatrix} \rmxt+\rmxz & \rmxx+i\rmxy \\ \rmxx-i\rmxy & \rmxt-\rmxz \end{pmatrix} \in \hermtwo.
	\]
Then 
	\[
		\langle X, X \rangle = - \det{X}.
	\]
An arbitrary $F \in \sltc$ acts on $\mathbb{L}^4$ as an orientation preserving isometry via the action
	\begin{equation}\label{Eq:202503220510PM}
		\hermtwo \ni X   \mapsto F X F^\star \in \hermtwo
	\end{equation}
where $F^\star$ is the conjugate transpose of $F$.
In fact, $\sltc$ is a two-to-one covering of the group of orientation-preserving and origin-fixing isometries of $\lfour$.

The submanifolds of $\lfour$ we are interested in are
\begin{align*}
	\ithree &:= \{ X \in \hermtwo : \rmxt-\rmxz=0 \}, \\
	\qthreep &:= \{ X \in \hermtwo : \langle X, X \rangle = 0, \tr{X} >0 \}, \\
	\qthreem &:= \{ X \in \hermtwo : \langle X, X \rangle = 0, \tr{X} <0 \}.
\end{align*}
In particular, the action of $\sltc$ given by \eqref{Eq:202503220510PM} acts as orientation preserving isometries of $\qthreep$. 
Conversely, any orientation preserving isometry of $\qthreep$ can be described by this action.

\begin{example}\label{exam:rotation}
	Rotations in $\qthreep$ can be described as isometries that fix a $2$-dimensional subspace of $\lfour$ (see, for example, \cite{docarmo_RotationHypersurfacesSpaces_1983, CKLLY1}).
	Normalizing the $2$-dimensional subspace based on the induced metric, they can be described by 
		$$
			D_1(\mu) := \begin{pmatrix} \mu & 0 \\  0 & \frac{1}{\mu} \end{pmatrix},\quad
			D_2(\mu) := \begin{pmatrix} 0 & \mu \\  -\frac{1}{\mu} & 0 \end{pmatrix},\quad
			P(\mu) := \begin{pmatrix} 1 & \mu \\  0 & 1 \end{pmatrix}.\
		$$
	for some $\mu \in \mathbb{C}$.
\end{example}

\subsection{Curve theory}\label{sect:curve}
Given a unit-speed regular curve $\gamma : I \to \qthreep \subset \lfour$ given on an interval $I$, define
	\[
		\kappa := -\frac{1}{2} \langle \gamma'', \gamma'' \rangle, \quad \T := \gamma', \quad\text{and}\quad \N := \kappa \gamma - \gamma''. 
	\]
Then there exist uniquely a function $\tau$ and a vector field $\B$ along $\gamma$ which form a null basis of $\lfour$ with $\det(\gamma, \T, \N, \B) < 0$:

	\begin{gather*}
		\langle X,Y \rangle \\
		\begin{array}{|c|c|c|c|c|}
			\hline
			\text{\diagbox{$X$}{$Y$}} &\gammab	& \T		& 	\N			& \B		\\ \hline
			\gammab	&	0			& 	0			& 	1	& 0		\\ \hline
			\T				&     	0		& 	1 			& 	0 	& 0		\\ \hline
			\N				&	1 	& 	0	& 	0 			& 0		\\ \hline
			\B				&	0 	& 	0	& 	0 			& 1		\\ \hline
		\end{array}
		\end{gather*}
		
Using such $\{\gamma, \T, \N, \B\}$ as a moving frame along the curve $\gamma$, the Frenet equations are given by
	\begin{equation}\label{Eq:202503190222PM}
		\begin{pmatrix} \gamma' \\  \T'  \\  \N' \\  \B' \end{pmatrix}
		=
		\begin{pmatrix} 0 & 1 & 0 & 0  \\  \kappa & 0 & -1 & 0  \\ 0 & -\kappa & 0 & -\tau \\ \tau & 0 & 0 & 0  \end{pmatrix}
		\begin{pmatrix} \gamma \\  \T \\  \N \\  \B \end{pmatrix}, 
	\end{equation}
The values  $\kappa$ and $\tau$ are referred to as the \emph{cone curvature} and the \emph{cone torsion}, respectively.
See \cite{Liu, Liu3} for details.

\subsection{Surface theory}
Given an immersion $X: \mcU \subset \mathbb{R}^2 \to \qthreep$ with coordinates $(u,v) = (u^1, u^2) \in \mathcal{U}$, there is a unique map $\G : \mcU \to \qthreem$ which satisfies
	\[
		\langle \G, \G \rangle = \langle \G, X_u \rangle = \langle \G, X_v \rangle = \langle \G, X \rangle -1 = 0,
	\]
which is called the \emph{lightlike Gauss map} of $X$ \cite{izumiya_LightconeGaussMap_2004} (also referred to as the \emph{associated surface} \cite{Liu4}).
The first and the second fundamental forms of $X$ are then given by
	\[
		\mathbf{g} :=   \langle X_i, X_j \rangle \dif{u}^i \dif{u}^j , \quad
		\mathbf{A} :=  \langle \G, X_{ij} \rangle \dif{u}^i \dif{u}^j,
	\]
respectively, from which the definition of mean curvature and the (extrinsic) Gauss curvature for $X$ follows:
	\[
		\mathrm{H} := \frac{1}{2} \tr{(\mathbf{g}^{-1}\mathbf{A})}, \quad
		\mathrm{K} := \det{(\mathbf{g}^{-1}\mathbf{A})}.
	\]
When a surface has $\mathrm{H} \equiv c \neq 0$ for some constant $c \in \mathbb{R}$, then the surface is called a \emph{constant mean curvature (CMC) surface}, while if it has $\mathrm{H} \equiv 0$, then the surface is referred to as a \emph{zero mean curvature (ZMC) surface}.

When a surface is conformally parametrized, we will introduce the complex structure via $z = u + i v$.

An immersion $X: \mathcal{U} \to \qthreep$ admits three different representations, each of which we will make use in this manuscript:
An immersion $X$ can be viewed as a graph of any function $f : \mathcal{U} \to \qthreep$ over the surface
	\[
		(u,v) \mapsto \begin{pmatrix} u^2+v^2 & u +iv \\ u-iv & 1 \end{pmatrix}
	\]
so that
	\[
		X(u,v) = e^{f(u,v)} \begin{pmatrix} u^2+v^2 & u +iv \\ u-iv & 1 \end{pmatrix}.
	\]
When we represent an immersion $X$ as a graph using the function $f$, we will denote it by $X_f$.

On the other hand, using the fact that the map
	\[
	\mathbb{C}^2\setminus\{(0,0)\} \ni \begin{pmatrix} z \\ w \end{pmatrix}  \mapsto
	 \begin{pmatrix} z \\ w \end{pmatrix} \begin{pmatrix} z \\ w \end{pmatrix}^\star 
	 = \begin{pmatrix} z\bar{z} & z\bar{w}\\ \bar{z}w & w\bar{w}\end{pmatrix} \in \qthreep
	\]
is onto, we can represent an immersion via
	\[
		X = \begin{pmatrix} A \\ C \end{pmatrix} \begin{pmatrix} A \\ C \end{pmatrix}^\star =: \varphi \varphi^\star
	\]
for some complex-valued functions $A$ and $C$ defined on $\mathcal{U}$.
We will refer to such $\varphi : \mathcal{U} \to \mathbb{C}^2$ as the \emph{lift of $X$ into $\mathbb{C}^2$}.

Finally, for any $F : \mathcal{U} \to \mathrm{SL}(2,\mathbb{C})$, we can also represent an immersion $X$ via
	\[
		X = F \begin{pmatrix} 1 & 0 \\ 0 & 0 \end{pmatrix} F^\star.
	\]
We will refer to such $F$ as the \emph{lift of $X$ into $\mathrm{SL}(2,\mathbb{C})$}.
Using this last characterization, we introduce the following notion:
\begin{definition}
	We say two immersions $X$ and $\tilde{X}$ are \emph{equal up to isometries and homothety of $\qthreep$} if they satisfy $X(s,t) = r F X(s,t) F^\star$ for some $r \in \mathbb{R}^+$ and $F \in \sltc$, and denote this as
	\[
		X \simeq \tilde{X}.
	\]
\end{definition}

\begin{remark}\label{Rmk:202503160842PM}
	For $c \in \mathbb{C}\setminus\{0\}$, we have
		\[
			\begin{pmatrix} c & 0 \\ 0 & c^{-1} \end{pmatrix}
			\begin{pmatrix} 1 & 0 \\ 0 & 0 \end{pmatrix}
			\begin{pmatrix} c & 0 \\ 0 & c^{-1} \end{pmatrix}^\star
			=
			c\bar{c}\begin{pmatrix} 1 & 0 \\ 0 & 0 \end{pmatrix}.
		\]
	If $F_c := F \begin{pmatrix} c & 0 \\ 0 & c^{-1} \end{pmatrix}$, then
		\begin{gather*}
			X_c := F_c  \begin{pmatrix} 1 & 0 \\ 0 & 0 \end{pmatrix} F_c^\star 
				= c\bar{c} F  \begin{pmatrix} 1 & 0 \\ 0 & 0 \end{pmatrix} F^\star  = c\bar{c} X,
			\quad \dif{F}_c F_c^{-1}  = \dif{F} F^{-1}.
\end{gather*}
In other words, homothetic images of $X$ can be obtained via $F_c$.
\end{remark}

\subsection{Representations of ZMC immersions}\label{SubS:202503220849PM}
Suppose $X = X_f$ is an immersion regarded as a graph using some function $f : \mathcal{U} \to \qthreep$.
Then the mean curvature of $X_f$ can be calculated as
	\begin{equation}\label{Eq:202503160928PM}
		\mathrm{H} = \frac{1}{2} e^{-2f(u,v)} (f_{uu}(u,v) + f_{vv}(u,v)).
	\end{equation}
\begin{example}[A family of CMC surfaces in $\qthreep$]
Let us consider the surface $X_f : \mathcal{U} \to \qthreep$ given via the function $f : \mathcal{U} \to \qthreep$
	\[
		f(u,v) := d + \ln{\operatorname{sech}{(a u+bv+c)}}
	\]
for some constants $a,b,c,d \in \mathbb{R}$ with $a^2+b^2 \not= 0$.
Then one can directly check that the mean curvature of $X_f$ is
	\[
		\mathrm{H} = - e^{-2d}(a^2+b^2)/2.
	\]
\end{example}
If $\mathrm{H} \equiv 0$ so that $X_f$ is a ZMC surface, then $f$ must be harmonic, so $f = \varphi + \bar{\varphi}$ for some holomorphic function $\varphi$ in a simply connected domain.
Then, after a conformal change of parameters if necessary, we see the following:
\begin{lemma}[\cite{Liu2}]
	$X : \mcU \to \qthreep$ has ZMC if and only if
		\[
			X(z) = \varphi(z) \varphi(z)^\star, \quad \varphi(z) 
			:= \begin{pmatrix} A(z)  \\ C(z) \end{pmatrix}
		\]
	for some holomorphic functions $A$ and $C$. 
\end{lemma}

Now we would like to see how the lift of a ZMC immersion $X$ to $\mathbb{C}^2$ induces the lift into $\mathrm{SL}(2,\mathbb{C})$:
A holomorphic $F \in C^\omega(\mcU, \sltc)$ is called \emph{null} if $\det{F_z}=0$.
Given an arbitrary holomorphic $\varphi(z) := \begin{psmallmatrix} A(z) \\ C(z) \end{psmallmatrix}$ that is the lift of a ZMC immersion to $\mathbb{C}^2$, one can find a null-holomorphic
	\[
		F = \begin{pmatrix} A & B \\ C & D \end{pmatrix} \in C^\omega(\mcU, \sltc)
	\]
such that 
	\[
		F \begin{pmatrix} 1 & 0 \\ 0 & 0 \end{pmatrix} F^\star 
			=  \begin{pmatrix} A \\ C  \end{pmatrix} \begin{pmatrix} A \\ C  \end{pmatrix}^\star
	\]
via
	\begin{equation}\label{Eq:202503230200PM}
		F = 
			\begin{pmatrix} A & 0 \\ C & A^{-1} \end{pmatrix}
			\begin{pmatrix} 1 & -E \\ 0 & 1 \end{pmatrix},
			\quad
		E:=\int \frac{ ((1/A)')^2}{ (C/A)'} \dif{z}.
	\end{equation}
For a given null-holomorphic $F$ we define meromorphic functions $G, g$ and holomorphic $1$-forms $\Omega, \omega$ by
	\begin{equation}\label{Eq:202501250902AM}
		\dif{F} F^{-1} = \begin{pmatrix} G & -G^2 \\ 1 & -G \end{pmatrix} \Omega, \quad
		F^{-1}\dif{F}  = \begin{pmatrix} g & -g^2 \\ 1 & -g \end{pmatrix} \omega.
	\end{equation}
This is the \emph{Weierstrass representation} of ZMC surfaces in $\qthreep$ \cite{Liu2, Liu4, seo_ZeroMeanCurvature_2021, pember_WeierstrasstypeRepresentations_2020}, and $(g, \omega)$ is called the \emph{Weierstrass data}.
We call $G$ and $g$ the hyperbolic Gauss map and the secondary Gauss map, respectively.
The Hopf differential is
	\[
		Q := \Omega \dif{G} = \omega \dif{g}.
	\]
For $\lambda \in \mathbb{C}\setminus\{0\}$, the change of the Weierstrass data
	\begin{equation}\label{Eq:202502051202PM}
		(g, \omega) \mapsto (g, \lambda \omega)
	\end{equation}
induces a transformation where the metric, Hopf differential, and the second fundamental form transform as follows:
	\[
		\mathbf{g} \mapsto |\lambda|^2 \mathbf{g},	\quad
		Q \mapsto \lambda Q, 				\quad
		\mathbf{A} = 2 \operatorname{Re} Q \mapsto  \mathbf{A}_\lambda = 2 \operatorname{Re} \lambda Q.
	\]
When $|\lambda|^2 = 1$, this gives an isometric deformation of a ZMC surface, commonly referred to as the \emph{associated family of ZMC surfaces}.

\subsection{Catenoids in $\qthreep$}\label{section2.3}
As examples of ZMC surfaces, we review catenoids in $\qthreep$, ZMC surfaces that are invariant under rotations \cite{CKLLY1}.
For real constants $a,b,c$, consider
	\begin{equation}\label{Eq:202503170717PM}
		\varphi^E_a(z) := e^{iaz} \begin{pmatrix} e^{iz} \\ e^{-iz} \end{pmatrix}, \quad
		\varphi^H_b(z) := e^{ibz} \begin{pmatrix} e^{z} \\ e^{-z} \end{pmatrix}, \quad
		\varphi^P_c(z) := e^{icz} \begin{pmatrix} z \\ 1 \end{pmatrix}.
	\end{equation}
They are lifts to $\mathbb{C}^2$ of elliptic, hyperbolic, parabolic catenoids in $\qthreep$, respectively. 
Their lifts to $\sltc$ are
	\begin{align*}
		F^E_a(z) &:= \begin{pmatrix} e^{i (a+1) z} &  -\frac{(a+1)^2 }{4 a} e^{-i (a-1)    z} \\  e^{i (a-1) z} &    -\frac{(a-1)^2 }{4 a} e^{-i (a+1)    z}\end{pmatrix}, \\
		F^H_b(z) &:= \begin{pmatrix} e^{ (i b+1) z} & -\frac{i    (b-i)^2 }{4 b} e^{-(i b-1) z}   \\  e^{(ib-1) z} & -\frac{i    (b+i)^2 }{4 b} e^{-(i b+1) z} \end{pmatrix}, \\
		F^P_c(z) &:= \begin{pmatrix} z e^{i c z} & -\frac{1}{2} (2+i c z) e^{-i c z}  \\  e^{i c z} & -\frac{1}{2} i c  e^{-i c z}\end{pmatrix}.
	\end{align*}
Their Weierstrass data can be written as 
	\[
		g^E(w)=g^H(w)=g^P(w)=w
	\]
and
	\begin{equation}\label{Eq:202502070726AM}
		\omega^E = \left( -\frac{1}{4} + \frac{1}{4a^2} \right)	\frac{\operatorname{d}\!w}{w^2},	\quad
		\omega^H = \left( -\frac{1}{4} - \frac{1}{4b^2} \right)	\frac{\operatorname{d}\!w}{w^2},	\quad
		\omega^P = -\frac{1}{4} \frac{\operatorname{d}\!w}{w^2}.
	\end{equation}
\begin{remark}
	Note that there is only one parabolic catenoid up to isometries and homothety of $\qthreep$. 
	For CMC-1 surfaces in de Sitter three-space, a similar fact has been observed in \cite{fujimori_AnalyticExtensionsConstant_2022}.
\end{remark}

For elliptic catenoids, the time component $\rmxt(u,v)$ of $X^E_a(u,v) := F^E_a(z)(F^E_a(z))^\star$ satisfies $\rmxt(u,v) = e^{-2av} \cosh{2v}. $
So, for any $u$,
	\[
		\begin{array}{|c|c|c|c|} \hline
			&	a < -1 & -1 < a < 1 & 1 < a \\				\hline
			\lim_{v \to \infty} \rmxt(u,v) & +\infty & +\infty & 0 \\		\hline
			\lim_{v \to -\infty} \rmxt(u,v) & 0 & +\infty & +\infty 	\\	\hline
		\end{array}
	\]
In other words, the two ends go to the future boundary if and only if $0<-\tfrac{1}{4}+\tfrac{1}{4a^2}$. 
One end goes to the future boundary and the other end goes to the origin if and only if  $-\tfrac{1}{4}<-\tfrac{1}{4}+\tfrac{1}{4a^2}<0$.

\section{Geodesics of $\qthreep$}\label{Section3}
As the induced metric on $\qthreep$ is degenerate, we will consider geodesics in $\qthreep$ by using characterizations of geodesics in better known quadrics of $\lfour$: hyperbolic $3$-space $\hthree$ and de Sitter $3$-space $\dsthree$.
In both cases, geodesics are obtained as intersections of each corresponding space and a two dimensional plane in $\lfour$ which passes through the origin; thus, they are curves in a totally geodesic surface, i.e.\ a totally umbilic surface with vanishing mean and (extrinsic) Gaussian curvature, whose geodesic curvature vanishes.

\subsection{Planes in $\qthreep$}
We first define planes as follows:
\begin{definition}
	A \emph{plane} is an immersion in $\qthreep$ that is totally geodesic, that is, the second fundamental form $\mathbf{A}$ vanishes everywhere.
\end{definition}
We will characterize planes within the class of totally umbilic surfaces in $\qthreep$: they are given by intersections with affine $3$-planes, namely, 
	\[
		S[M,q] := \{ X \in \qthreep : \langle X,M\rangle = q\}
	\]
for some constant $M \in \lfour$ and $q \in \mathbb{R}\setminus \{0\}$.
Then it can be directly checked that the lightlike Gauss map of $S[m,q]$ is given by
	\[
		\G = -\frac{1}{2q^2} \langle M,M\rangle X + \frac{1}{q} M
	\]
so that its mean and Gaussian curvature satisfies
	\[
		\mathrm{H} = \frac{1}{2q^2} \langle M,M\rangle, \quad \mathrm{K} = \frac{1}{4q^4}\langle M,M \rangle^2.
	\]
Thus:
\begin{lemma}
	An immersion is part of a plane if and only if it is part of $P[M,q]$ defined via
		\[
			P[M,q] := \{ X \in \qthreep : \langle X,M\rangle = q, \langle M, M \rangle = 0\}.
		\]
\end{lemma}

\subsection{Geodesics}
Now that we have the notion of a plane, we define geodesics as follows:
\begin{definition}
	A regular curve in a plane of $\qthreep$ is called a \emph{geodesic} if its geodesic curvature vanishes in the plane.
\end{definition}

To obtain an explicit formulation of geodesics, let us suppose that a plane is given via $P[M,q]$ with $\langle M,M\rangle = 0$; we will now calculate the metric induced on $P[M,q]$ from the ambient space by constructing a coordinate chart.
Choosing any vector $\tilde{M} \in \lfour$ such that $\langle M, \tilde{M}\rangle = q \neq 0$, note that $\mathrm{span}\{M, \tilde{M}\}$ must have signature $(- +)$.
Thus,
	\[
		\mathfrak{R} := \mathrm{span}\{M, \tilde{M}\}^\perp \cong \mathbb{E}^2
	\]
where $\mathbb{E}^2$ denotes the usual Euclidean $2$-plane.

Let us define a bijection $\psi : \mathfrak{R} \to P[M,q]$ by
	\[
		\psi(Y) = Y + \tilde{M} - \frac{1}{2q} \langle Y,Y\rangle M
	\]
with inverse
	\[
		\psi^{-1}(X) = X - \tilde{M} - \frac{1}{q}\langle X, \tilde{M}\rangle M.
	\]
Viewing $\psi$ as a coordinate chart, let $\gamma : I \to P[m,q]$ be a unit speed curve with $Y : I \to \mathfrak{R}$ such that
	\[
		\gamma = \psi \circ Y.
	\]
Then since
	\[
		\gamma' = (Y + \tilde{M} - \tfrac{1}{2q} \langle Y,Y\rangle M)' = Y' - \frac{1}{2q}\langle Y, Y'\rangle M,
	\]
it must follow
	\[
		\langle \gamma', \gamma'\rangle = \langle Y', Y' \rangle.
	\]
Therefore, any plane $P[m,q]$ is isometric to $\mathfrak{R} \cong \mathbb{E}^2$, and thus $\gamma : I \to P[m,q]$ is a curve in $P[m,q]$ with vanishing geodesic curvature if and only if $Y := \psi^{-1}\circ \gamma : I \to \mathfrak{R} \cong \mathbb{E}^2$ is a line.

To obtain explicit parametrizations of geodesics, let $Y: I \to \mathfrak{R}$ be a line parametrized by arc-length so that
	\[
		Y(s) = Vs + W
	\]
for any constant $V,W \in \mathfrak{R} = \mathrm{span}\{M, \tilde{M}\}^\perp$ with $\langle V, V \rangle = 1$.
We then have
	\begin{align*}
		\gamma(s) = \psi \circ Y(s) &= Vs + W + \tilde{M} - \tfrac{1}{2q}(s^2 + 2 \langle V, W\rangle s + \langle W,W\rangle)M \\
			&= -\tfrac{1}{2q}M s^2 + (V - \tfrac{1}{q} \langle V, W\rangle M) s + W + \tilde{M} -\tfrac{1}{2q} \langle W, W \rangle M\\
			&:= \frac{1}{2}\vec{a} s^2 + \vec{b} s + \vec{c}.
	\end{align*}
Then we can check directly that
	\begin{equation}\label{eqn:geodesicCond}
		\langle \vec{a}, \vec{a} \rangle = \langle \vec{a}, \vec{b} \rangle = \langle \vec{b},\vec{c}\rangle = \langle \vec{c}, \vec{c} \rangle = 0, \quad
		\langle \vec{a},\vec{c}\rangle= -\langle \vec{b},\vec{b} \rangle = 1.
	\end{equation}

On the other hand, let $\gamma : I \to \qthreep$ given via
	\[
		\gamma(s) = \frac{1}{2}\vec{a} s^2 + \vec{b} s + \vec{c}
	\]
where $\vec{a},\vec{b},\vec{c} \in \lfour$ satisfies \eqref{eqn:geodesicCond}.
Then we have
	\[
		\langle \gamma(s), \vec{a} \rangle = -1,
	\]
so that $\gamma$ is a curve in the plane $P[\vec{a}, -1]$.
Therefore, using $\psi^{-1}: P[\vec{a}, -1] \to \mathfrak{R}$, we write $Y := \psi^{-1} \circ \gamma$ with $\vec{c} = \tilde{M}$ as
	\[
		Y = \tfrac{1}{2}\vec{a} s^2 + \vec{b} s + \vec{c} - \vec{c} + \langle \tfrac{1}{2}\vec{a} s^2 + \vec{b} s + \vec{c}, \vec{c} \rangle \vec{a} = \vec{b} s
	\]
so that $Y$ is a line.

Summarizing:
\begin{theorem}\label{thm:geodesic}
	A curve $\gamma : I \to \qthreep$ is a geodesic of $\qthreep$ if and only if there exists some $\vec{a}, \vec{b}, \vec{c} \in \lfour$ satisfying
		\begin{equation}\label{Eq:202501110701AM}
			\begin{array}{|c|c|c|c|}
				\multicolumn{4}{c}{\langle X,Y \rangle }  \\  \hline
				\text{\diagbox{$X$}{$Y$}} & \vec{a}	& \vec{b} & 	\vec{c} 	\\ \hline
				\vec{a}	&	0		& 	0	& 	-1 		\\ \hline
				\vec{b}	&     0		& 	1 	& 	0 		\\ \hline
				\vec{c}	&	-1		& 	0	& 	0 		\\ \hline
				\end{array}
		\end{equation}
	such that
		\begin{equation}\label{eqn:geodesicEq}
			\gamma(s) = \frac{1}{2}\vec{a} s^2 + \vec{b} s + \vec{c}.
		\end{equation}
\end{theorem}

Using the parametrization, we can also deduce the following:
\begin{lemma}\label{lemma:congruent}
	All geodesics are congruent to each other.
\end{lemma}
\begin{proof}
	Let $\vec{a}, \vec{b}, \vec{c}$ be arbitrary vectors of $\lfour$ which satisfy \eqref{Eq:202501110701AM}.
	By applying an isometry of $\qthreep$, we may assume that  $\vec{c}=\begin{pmatrix} 0& 0 \\ 0 & 1 \end{pmatrix}$. 
	Then \eqref{Eq:202501110701AM} implies that
		\[
			\vec{a} = \begin{pmatrix} 2 & w \\ \bar{w} & \frac{1}{2}w\bar{w} \end{pmatrix}, \quad
			\vec{b} = \begin{pmatrix} 0 & e^{i\theta} \\ e^{-i\theta} & \frac{1}{2}(w e^{-i\theta}+ \bar{w}e^{i\theta}) \end{pmatrix}
		\]
	for some $w \in \mathbb{C}$ and $\theta \in \mathbb{R}$.
	Let $F := \begin{pmatrix} e^{-i\theta/2} & 0 \\ -\frac{1}{2} \bar{w} e^{i\theta/2} & e^{i\theta/2} \end{pmatrix} \in \sltc$, then
		\[
			F \vec{a} F^\star = \begin{pmatrix} 2 & 0  \\ 0 & 0 \end{pmatrix}, \quad
			F \vec{b} F^\star = \begin{pmatrix} 0 & 1  \\ 1 & 0 \end{pmatrix}, \quad
			F \vec{c} F^\star = \begin{pmatrix} 0 & 0  \\ 0 & 1 \end{pmatrix},
		\]
	and thus the claim follows.
\end{proof}

\subsection{Geodesics as space curves}
To make connection with the curve theory reviewed in Section~\ref{sect:curve}, let $\gamma$ be a geodesic given by the unit speed parametrization as in Theorem~\ref{thm:geodesic}, so that
	\[
		\gamma'(s) = \vec{a} s + \vec{b}, \quad \gamma''(s) = \vec{a}.
	\]
Thus the cone curvature of $\gamma$ vanishes:
	\[
		\kappa = -\frac{1}{2} \langle \gamma'', \gamma'' \rangle = 0.
	\]
On the other hand, since
	\[
		\mathrm{span}\{ \gamma, \T, \N\} = \mathrm{span}\{ \gamma, \gamma', \gamma''\} = \mathrm{span}\{ \vec{a}, \vec{b}, \vec{c}\}
	\]
is constant in $s$, we have that $\B: I \to \mathrm{span}\{ \gamma, \T, \N\}^\perp$ is a constant vector, and hence the cone torsion $\tau$ also vanishes.

To consider the converse, let us assume that $\gamma : I \to \qthreep$ is a unit speed curve with $\kappa = \tau = 0$.
Then the Frenet equation \eqref{Eq:202503190222PM} implies $\gamma''' = T'' =  - N' = \vec{0}$.
Hence, $\gamma''$ is a constant vector, say $\vec{a} \in \lfour$.
Since $\kappa=0$ if and only if $\gamma''$ is lightlike, we have $\langle \vec{a}, \vec{a} \rangle = 0$.

Integrating once with respect to $s$, we obtain $\gamma'(s) = \vec{a} s + \vec{b}$ for some $\vec{b} \in \lfour$, and since $\gamma$ is unit speed,
	\[
		1 = \langle \gamma', \gamma' \rangle = \langle \vec{a} s + \vec{b}, \vec{a} s + \vec{b} \rangle = 2 \langle \vec{a}, \vec{b} \rangle s + \langle \vec{b}, \vec{b} \rangle,
	\]
allowing us to conclude that
$\langle \vec{a}, \vec{b} \rangle =0, \langle  \vec{b},  \vec{b} \rangle =1$.

Integrating once more with respect to $s$, we may write
	\[
		\gamma(s) = \frac{1}{2}\vec{a} s^2 + \vec{b} s + \vec{c}
	\]
for some $\vec{c} \in \lfour$.
The fact that $\gamma$ takes values in $\qthreep$ tells us
	\[
		0 = \langle \gamma(s), \gamma(s) \rangle = \langle \vec{a}, \vec{c} \rangle s^2 + s^2 +  2 \langle \vec{b}, \vec{c} \rangle s + \langle \vec{c}, \vec{c} \rangle,
	\]
so that $\langle \vec{a}, \vec{c} \rangle = -1$, $\langle \vec{b}, \vec{c} \rangle =0$, and $\langle \vec{c}, \vec{c} \rangle =0$.

Thus we conclude:
\begin{theorem}
	A unit speed curve is a geodesic if and only if its cone curvature and cone torsion vanishes.
\end{theorem}

\begin{remark}
In Riemannian manifolds, a geodesic is determined by its initial position and initial velocity.
However, that is not the case in $\qthreep$.
For example, let $\gamma$ be given as in \eqref{eqn:geodesicEq} and take $\vec{a}$, $\vec{b}$, and $\vec{c}$ as in the proof of Lemma~\ref{lemma:congruent} with $\theta=0$.
Then any arbitrary choice of $w=i\mu \in i\mathbb{R}$ gives geodesics all with the same initial position and initial velocity.
This is due to the freedom of choice of $\tilde{M}$.
\end{remark}

\section{Two examples of ruled ZMC surfaces $\qthreep$}\label{Subsec:202503230110PM}
In this section, we will examine two important ruled ZMC surfaces in $\qthreep$, with an eye on obtaining a complete classification of all ruled ZMC surfaces in $\qthreep$.
\subsection{Screw motions in $\qthreep$}\label{sect:screwMotion}
One of the most important example of a ruled minimal surface in Euclidean space is a helicoid, obtained by applying a certain screw motion to a geodesic.
Our first step in the classification for ruled ZMC surface in $\qthreep$ is to mimic the Euclidean case, and apply a certain screw motion to a geodesic, and see if the resulting surface has zero mean curvature.
Thus, we first devote our attention to screw motions in $\qthreep$, by first examining isometries that form a one-parameter subgroup under composition:
\begin{fact}
	Let $\varphi : (-\epsilon, \epsilon) \to \sltc$ be a one-parameter subgroup of isometries of $\qthreep$, so that 
		\[
			\varphi(s+t) = \varphi(s) \varphi(t).
		\]
	Hence $\varphi(s)^{-1} ( \varphi(s+t) - \varphi(s)) = \varphi(t) - \varphi(0)$, which implies
		\[
			\varphi(s)^{-1} \varphi(s)' = \varphi(0)' \in \lasltc. 
		\]
	Then $\varphi(s) = e^{s \varphi(0)'}$.
	For $a,b,c \in \mathbb{C}$ with $\Delta := \sqrt{-(a^2+bc)} \not= 0$,
		\[
			\exp\left( s \begin{pmatrix} a & b \\ c & -a \end{pmatrix} \right)
			= 
			\begin{pmatrix} 
				\cos{\Delta s} + \frac{a}{\Delta} \sin{\Delta s} & \frac{b}{\Delta} \sin{\Delta s}  \\
				\frac{c}{\Delta} \sin{\Delta s}  &  \cos{\Delta s} - \frac{a}{\Delta} \sin{\Delta s}
			\end{pmatrix}.
		\]
\end{fact}

So let us consider all cases of $A \in \lasltc$, so that $A^2 = -(\det{A}) I_2$, generating such one-parameter subgroups, where the rotation matrices of Example~\ref{exam:rotation} will be used.

\textbf{Case 1.} Suppose that $\det{A} \not= 0$.
Then $\lambda :=  \sqrt{-\det{A}}$ is an eigenvalue of $A$ and
	\[
		A = M \, \operatorname{diag}(\lambda, -\lambda) \, M^{-1}
	\]
for some $M \in \sltc$.
Then the one-parameter subgroup $s \mapsto e^{sA}$ is similar to $s \mapsto D_1(e^{\lambda s})$.
If $\lambda$ is purely imaginary, $\varphi$ are elliptic rotations. 
If $\lambda$ is real, $\varphi$ are hyperbolic rotations. 
If $\lambda$ is neither real nor purely imaginary, then $\varphi$ are the \emph{screw motions}.

\textbf{Case 2.} Suppose that $\det{A}=0$.
Then $A = \begin{pmatrix} \alpha\beta & -\alpha^2 \\ \beta^2 & -\alpha\beta \end{pmatrix}$ for some $\alpha, \beta \in \mathbb{C}$ and $e^{sA} = I + s A$.
Let
	\[
		\lambda := \sqrt{-i\alpha / \beta},\qquad h := 2i\alpha\beta \qquad \text{for }\ \alpha, \beta \in \mathbb{C}\ \backslash\ \{0\}.
	\]
Then
	\[
		D_2(\lambda) e^{sA} D_2(\lambda)^{-1} = \varphi(h s)
	\]
where
	\[
		\varphi(k) := \begin{pmatrix} 1+ i\frac{k}{2} & \frac{k}{2} \\ \frac{k}{2} & 1-i\frac{k}{2} \end{pmatrix}
		= B P(k) B^{-1},
		 \quad
		 B := \frac{1}{\sqrt{2}}\begin{pmatrix} 1 & -i \\ -i & 1\end{pmatrix}.
	\]
So the one-parameter subgroup $s \mapsto e^{sA}$ is similar to $s \mapsto P(sh)$ in $\sltc$ for some $h \in \mathbb{C}$. 
If $\alpha$ or $\beta$ is 0, it's trivial. Note that, if $h=r e^{i\theta}$, then 
	\[
		\begin{pmatrix} e^{-i\theta/2} & 0 \\ 0 & e^{i\theta/2} \end{pmatrix}
		\begin{pmatrix} 1 & s r e^{i\theta} \\ 0 & 1 \end{pmatrix}
		\begin{pmatrix} e^{-i\theta/2} & 0 \\ 0 & e^{i\theta/2} \end{pmatrix}^{-1}
		=
		\begin{pmatrix} 1 & r s \\ 0 & 1 \end{pmatrix}, 
	\]
so it is conjugate to a single parabolic rotation.
In conclusion, we obtain the following:
\begin{proposition}
	Any one-parameter subgroup of isometries of $\qthreep$  is similar to
		\[
			D_1(e^{\lambda s}) = \begin{pmatrix} e^{\lambda s} & 0 \\ 0 & e^{-\lambda s} \end{pmatrix},
			\quad\text{or}\quad
			P(rs) = \begin{pmatrix} 1 & r s \\ 0 & 1 \end{pmatrix},
		\]
	for some $\lambda \in \mathbb{C}$  and $r \in \mathbb{R}$.
\end{proposition}
Now that we have the notions of geodesics and screw motions, we define the following special classes of surfaces:
\begin{definition}
	Let $X : \mathcal{U} \to \qthreep$ be an immersion.
	\begin{itemize}
		\item If $X$ is invariant under screw motions, then $X$ is called a \emph{helicoidal surface}.
		\item If $X$ is foliated by geodesics, then $X$ is called a \emph{ruled surface}.
	\end{itemize}
\end{definition}

\subsection{Helicoids}\label{SubSec:202503231257PM}
Now let us consider the surface obtained by applying a screw motion to a geodesic:
	\begin{align*}
		X(u,v) &:= 
			\begin{pmatrix} e^{(a + ib)v} & 0 \\ 0 & e^{-(a + ib)v} \end{pmatrix}	 
			\begin{pmatrix} u^2 & u \\ u & 1 \end{pmatrix} 
			\begin{pmatrix} e^{(a + ib)v} & 0 \\ 0 & e^{-(a + ib)v} \end{pmatrix}^\star\\	
				&= \begin{pmatrix} e^{2av} u^2 & e^{2ibv} u \\ e^{-2ibv} u & e^{-2av} \end{pmatrix}.
	\end{align*}
Then one can calculate the lightlike Gauss map $\G$ and the corresponding fundamental forms to find that $\mathrm{H} \equiv 0$.
Thus, as in the case of minimal surfaces in Euclidean space, we define helicoids in $\qthreep$ as follows:
\begin{definition}
	For arbitrary real numbers $a,b$ with $b\not=0$, let $X$ be a ZMC surface defined by
		\[
			H^{a,b}(u,v)
			:= 
			\begin{pmatrix} e^{(a + ib)v} & 0 \\ 0 & e^{-(a + ib)v} \end{pmatrix}	 
				\begin{pmatrix} u^2 & u \\ u & 1 \end{pmatrix} 
				\begin{pmatrix} e^{(a + ib)v} & 0 \\ 0 & e^{-(a + ib)v} \end{pmatrix}^\star.		
		\]
	We call $X$ the \emph{standard helicoid} of $\qthreep$.
	Any surface that is congruent to $X$ up to isometries of $\qthreep$ are referred to as \emph{helicoids} (see left of Figure~\ref{fig:ruled1}).
\end{definition}
Now we note some important geometric facts about helicoids in $\qthreep$:
\begin{lemma}\label{lemma:helicoid}
	The following hold:
	\begin{itemize}
		\item	For any real numbers $a,b$ with $b\not=0$, the curve $v \mapsto H^{a,b}(\tfrac{1}{2\sqrt{a^2+b^2}},v)$ is a unit-speed helix and has constant cone curvature $\kappa = 2(a^2-b^2)$  and constant cone torsion $\tau = 4ab$.
		\item	The metric is given by $\mathbf{g} = \operatorname{d}\!u^2 + 4au \operatorname{d}\!u\operatorname{d}\!v + 4(a^2+b^2) u^2 \operatorname{d}\!v^2$, so $u$-parameter curves and $v$-parameter curves of $H^{a,b}$ make angles of constant magnitude not equal to $\tfrac{\pi}{2}$.
		\item	The metric is singular if $u=0$.
	\end{itemize}
\end{lemma}
\begin{proof}
	The proof follows from direct calculations.
\end{proof}

In particular, application of screw motion to geodesics does not result in orthogonal parametrization of the resulting helicoid, which makes the case of $\qthreep$ stand out from the cases of other space forms.
However, we can find a conformal reparametrization of the standard helicoid as follows:
If we let $\tilde{v} := 2bv$ and $\tilde{u} := 2av + \ln{u}$, then 
	\[
		\tilde{H}^{a,b}(\tilde{u},\tilde{v}) := H^{a,b}\left( e^{\tilde{u}-\tfrac{a}{b}\tilde{v}},\tfrac{1}{2b}\tilde{v} \right)
	\]
is conformally parametrized, and
\begin{equation}\label{Eq:202501280355AM}
	\begin{aligned}
	\tilde{H}^{a,b}(u,v) 
		&=  e^{-\frac{a}{b} v} 
			\begin{pmatrix} e^{i\frac{v}{2}} & 0 \\ 0 & e^{-i\frac{v}{2}} \end{pmatrix} 
			\begin{pmatrix} e^{2u} & e^u \\ e^u & 1 \end{pmatrix} 
			\begin{pmatrix} e^{i\frac{v}{2}} & 0 \\ 0 & e^{-i\frac{v}{2}} \end{pmatrix}^\star	\\
		&= \varphi^{\mathrm{Hel}}_{i \frac{a}{2b}}(u+iv) \varphi^{\mathrm{Hel}}_{i \frac{a}{2b}}(u+iv)^\star,  
		\qquad
		\varphi^{\mathrm{Hel}}_c(z) := e^{c z} \begin{pmatrix} e^z \\ 1 \end{pmatrix}.
	\end{aligned}
\end{equation}
Direct calculations show that the $v$-parameter curves do not have constant cone curvature, hence are not helices.

\begin{figure}
	\centering
	\begin{minipage}{0.49\linewidth}
		\centering
		\includegraphics[width=0.8\textwidth]{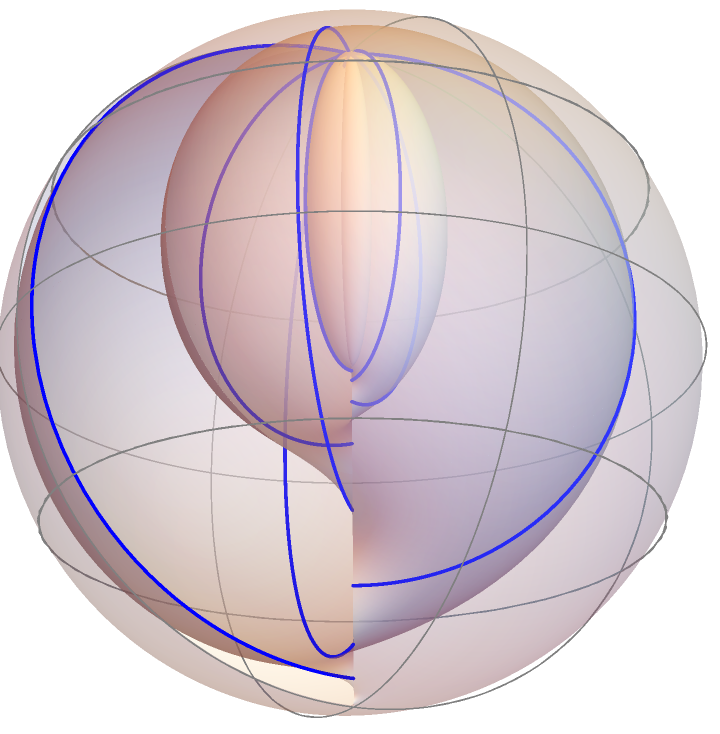}
	\end{minipage}
	\begin{minipage}{0.49\linewidth}
		\centering
		\includegraphics[width=0.8\textwidth]{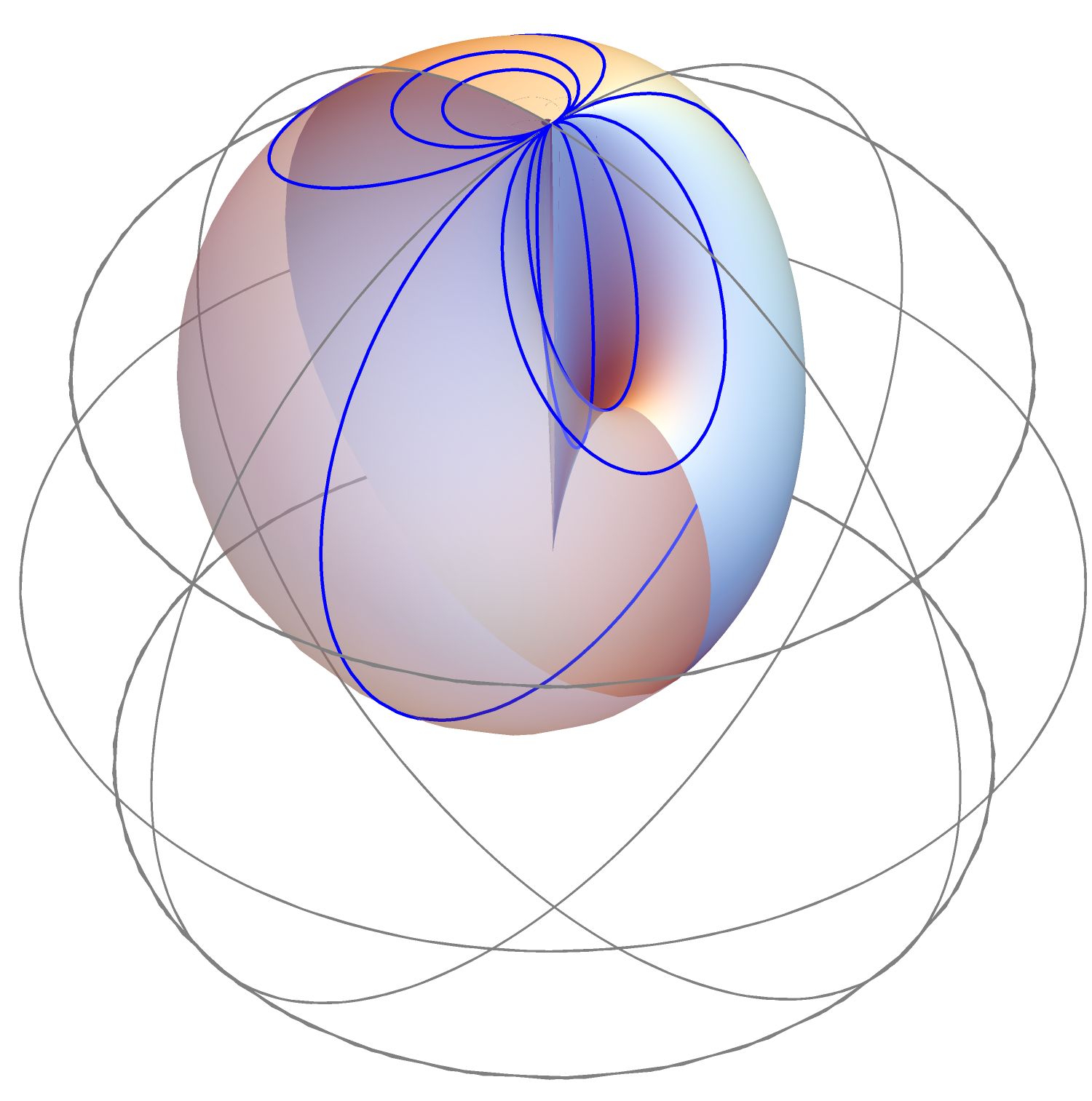}
	\end{minipage}
	\caption{A helicoid (left) and a parabolic catenoid (right) in the ball model of $\qthreep$ (cf.\ \cite{CKLLY1}). Blue curves represent geodesics.}
	\label{fig:ruled1}
\end{figure}

\subsection{Parabolic catenoids}\label{SubSec:202503231258PM}
We now give another important example that is a ruled ZMC surface in $\qthreep$: the parabolic catenoid reviewed in Section~\ref{section2.3}.
To check that parabolic catenoid is a ruled surface, we first observe that if $z=u+iv$ and $\varphi^P_c$ is from \eqref{Eq:202503170717PM}, then
	\begin{equation}\label{Eq:202503221038PM}
		C^P_c(u,v) := \varphi^P_c(u+iv)\varphi^P_c(u+iv)^\star
			=  e^{- 2 c v} 
			\begin{pmatrix} 1 & iv \\ 0 & 1 \end{pmatrix} 
			\begin{pmatrix} u^2 & u \\ u & 1 \end{pmatrix} 
			\begin{pmatrix} 1 & iv \\ 0 & 1 \end{pmatrix}^\star.
	\end{equation}
The fact that the parabolic catenoid is a ruled surface follows from the following reparametrization:
	\begin{equation}\label{Eq:202501280440}
		C^P_{c}(e^{2cv}u,v) 
			=\begin{pmatrix} 1 & iv \\ 0 & 1 \end{pmatrix}	 \begin{pmatrix} e^{ cv} & 0 \\ 0 & e^{-cv} \end{pmatrix}
			\begin{pmatrix} u^2 & u \\ u & 1 \end{pmatrix} 
			\begin{pmatrix} e^{ cv} & 0 \\ 0 & e^{-cv} \end{pmatrix}^\star
			\begin{pmatrix} 1 & iv \\ 0 & 1 \end{pmatrix}^\star	 
	\end{equation}
so that parabolic catenoid is a surface obtained by applying isometries to a geodesic.

\begin{remark}
	The map
		\[
			\mathbb{R} \ni c \mapsto p_c(v)  := \begin{pmatrix} 1 & iv \\ 0 & 1 \end{pmatrix}	 \begin{pmatrix} e^{ cv} & 0 \\ 0 & e^{-cv} \end{pmatrix}
		\]
in \eqref{Eq:202501280440} is not a one-parameter subgroup of isometries, that is,  $p_c(v_1 + v_2 )  \not= p_c(v_1) p_c(v_2)$.
Hence it is not a screw motion.
On the other hand, the map
	\[
		\mathbb{R} \ni c \mapsto q_c(v) := e^{-cv} \begin{pmatrix} 1 & iv \\ 0 & 1 \end{pmatrix} = \exp\begin{pmatrix} -cv & -iv \\ 0 & -cv \end{pmatrix} 
	\]
in \eqref{Eq:202503221038PM} induces a one-parameter group, that is, $q_c(v_1 + v_2) = q_c(v_1) q_c(v_2)$. 
However, $q_c(v)$ does not represent an isometry.
\end{remark}

\section{Helicoids and the associate family of catenoids}\label{Section4}
In the Euclidean case, helicoids and catenoids are related by an isometric deformation, called the associated family.
In this section, we examine whether helicoids in $\qthreep$ can also be found in the associated family of catenoids in $\qthreep$.

\subsection{Associated family of catenoids in $\qthreep$}
Recall from \eqref{Eq:202502051202PM} that a ZMC surface with Weierstrass data $(g, \omega)$ admits an isometric deformation known as the \emph{associated family} by a change in the Weierstrass data via
	\[
		(g, \omega) \mapsto (g, \lambda \omega)
	\]
for any unit complex constant $\lambda \in \mathbb{S}^1 \subset \mathbb{C}$.
Taking the Weierstrass data as
	\[
		(g, \omega_\delta) = \left(w, \delta \frac{\dif{w}}{w^2} \right),
	\]
for some $\delta \in \mathbb{R} \setminus \{0\}$, we note from \eqref{Eq:202502070726AM} that the resulting immersion $X_\delta$ is
	\begin{itemize}
		\item an elliptic catenoid if $\delta > -\frac{1}{4}$,
		\item a parabolic catenoid if $\delta = -\frac{1}{4}$, or
		\item a hyperbolic catenoid if $\delta < -\frac{1}{4}$.
	\end{itemize}
	
Therefore, any surface with Weierstrass data
	\[
		(g, \omega_\delta) = \left(w, \delta \frac{\dif{w}}{w^2} \right),
	\]
for any $\delta \in \mathbb{C}$ must be in the associated family of a catenoid, given by the Weierstrass data
	\[
		(g, \omega_{|\delta|}) = \left(w, |\delta| \frac{\dif{w}}{w^2} \right).
	\] 

\subsection{Helicoids and the associated family of catenoids}
Turning our attention to the conformally parametrized (standard) helicoid $\tilde{H}^{a,b}$ in \eqref{Eq:202501280355AM}, we first note that we may assume without loss of generality that $b = \frac{1}{2}$. 
Then the lift of $\tilde{H}^{a}$ to $\mathbb{C}^2$ is given in \eqref{Eq:202501280355AM} by
	\[
		\varphi^{\mathrm{Hel}}_a(z) := e^{ i a z} \begin{pmatrix} e^z \\ 1 \end{pmatrix}.
	\]
Thus we can find the lift $F^{a}$ to $\sltc$ using \eqref{Eq:202503230200PM}:
	\[
		F^{a}(z) =
			\begin{pmatrix} e^{\tfrac{z}{2}} & 0 \\ 0 & e^{-\tfrac{z}{2}} \end{pmatrix}
			\begin{pmatrix}
				1 & -\frac{(1+ia)^2}{1+2ia} \\
				 1 & -\frac{a^2}{1+2ia}
			\end{pmatrix}
			\begin{pmatrix} e^{i az} & 0 \\ 0 & e^{-i a z} \end{pmatrix}
			\begin{pmatrix} e^{\tfrac{z}{2}} & 0 \\ 0 & e^{-\tfrac{z}{2}} \end{pmatrix}.
	\]
This in turn allows us to find the Weierstrass data using \eqref{Eq:202501250902AM}, so that
	\begin{align*}
		(G, \Omega) &= \left( \frac{a-i}{a} e^z, a^2 e^{-z}\dif{z} \right) \\
		(g, \omega) &= \left(- \frac{a(1+ia)}{2a-i}  e^{(-1-2ia)z},  -e^{(1+2ia)z} \dif{z}\right).
	\end{align*}
Making a change of coordinate via $w := g$, we may normalize the Weierstrass data as
	\[
		(\tilde{g}, \tilde{\omega}) = \left(w, -\frac{a(a-i)}{(2a-i)^2}\frac{\dif{w}}{w^2} \right).
	\]
Therefore, we learn:
\begin{theorem}\label{thm:helicate}
	Every helicoid in $\qthreep$ is in the associated family of some catenoid in $\qthreep$.
\end{theorem}

However, the converse is not true, namely, there are catenoids that do not have helicoids in the associated family.
To see this, we note that for
	\[
		\delta(a) := -\frac{a(a-i)}{(2a-i)^2},
	\]
one can directly check that for real functions of $\rmx$ and $\rmy$ of $a$ given by
	\[
		\rmx(a) + i \, \rmy(a) := -\frac{1}{4} - \delta(a),
	\]
we have
	\[
		\left(\rmx(a)^2+\rmy(a)^2 \right)^2 + \frac{1}{4} \rmx(a) \left( \rmx(a)^2 + \rmy(a)^2 \right) - \frac{1}{64} \rmy(a)^2 = 0,
	\]
which is the formula for a cardioid.
As $a$ varies from $-\infty$ to $\infty$, the image starts from $(-1/4,0)$ and wraps around in the counterclockwise direction.

In particular, we learn that
	\[
		|\delta(a)| \in [0,\tfrac{1}{2\sqrt{3}}],
	\]
so that when $|c| > \tfrac{1}{2\sqrt{3}}$, the catenoids with Weierstrass data $(g, \omega_c) = \left(z, c \frac{\dif{w}}{w^2} \right)$, have no helicoids in its associated family.
We summarize:
\begin{theorem}\label{thm:cateheli}
Let $X$ be a ZMC catenoid in $\mathbb{Q}^3_+$ given by the Weierstrass data
	\[
		\mathrm{g}=\zeta,\qquad \omega = \frac{\delta}{\zeta^2}\operatorname{d}\!\zeta \qquad\text{for}\quad \delta \in \mathbb{R}^+.
	\]
Then the catenoids can be classified according to the different catenoids and helicoids they admit under their associated families (see also Figure~\ref{Fig:202502060736PM}):
\begin{itemize}
	\item $0<\delta<\frac{1}{4}${\upshape:} two elliptic catenoids and two helicoids (Fig.\ \ref{Fig:cat1}),
	\item $\delta=\frac{1}{4}${\upshape:} one elliptic catenoid, two helicoids, and one parabolic catenoid (Fig.\ \ref{Fig:cat2}),
	\item $\frac{1}{4}<\delta<\frac{1}{2\sqrt3}${\upshape:} one elliptic catenoid, one hyperbolic catenoid, four helicoids (Fig.\ \ref{Fig:cat3}),
	\item $\delta=\frac{1}{2\sqrt3}${\upshape:} one elliptic catenoid, one hyperbolic catenoid, two helicoids (Fig.\ \ref{Fig:cat4}),
	\item $\frac{1}{2\sqrt3}<\delta${\upshape:} one elliptic catenoid, one hyperbolic catenoid.
\end{itemize}
\end{theorem}

\begin{figure}
	\begin{center}
		\includegraphics[width=0.6\textwidth]{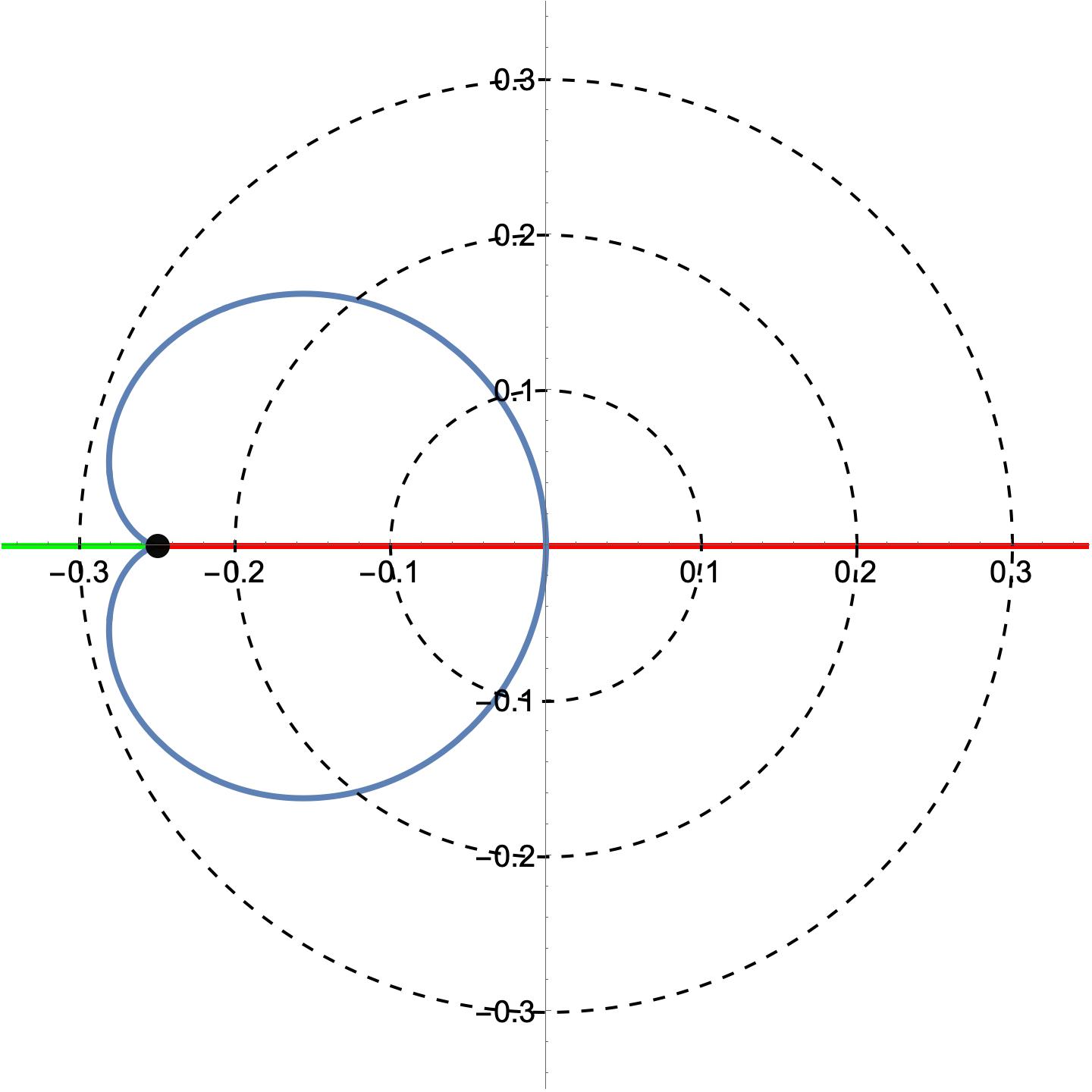}
	\end{center}
	\caption{Types of surface given by the Weierstrass data $(g,\omega) = (w, \delta \frac{\dif{w}}{w^2})$, depending on the value of $\delta \in \mathbb{C}$. The red line represents elliptic catenoids, black point represents parabolic catenoid, green line represents hyperbolic catenoids, and blue cardioid represents helicoids.}
	\label{Fig:202502060736PM}
\end{figure}

\begin{figure}
	\begin{center}
		\includegraphics[width=0.8\textwidth]{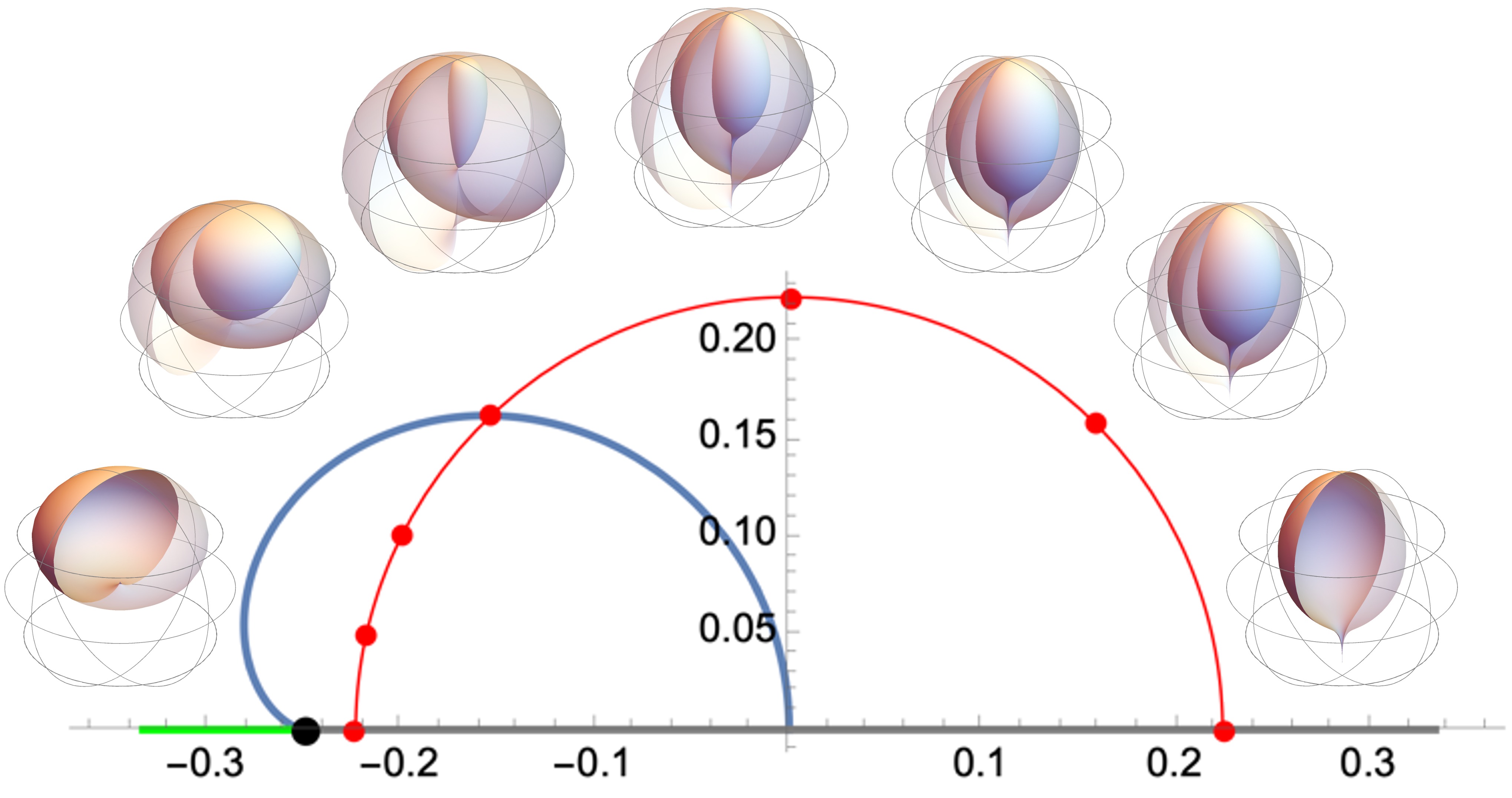}
	\end{center}
	\caption{Associated families of catenoid given by the Weierstrass data $(g,\omega) = (w, \frac{1}{2\sqrt{5}} \frac{\dif{w}}{w^2})$.}
	\label{Fig:cat1}
\end{figure}

\begin{figure}
	\begin{center}
		\includegraphics[width=0.8\textwidth]{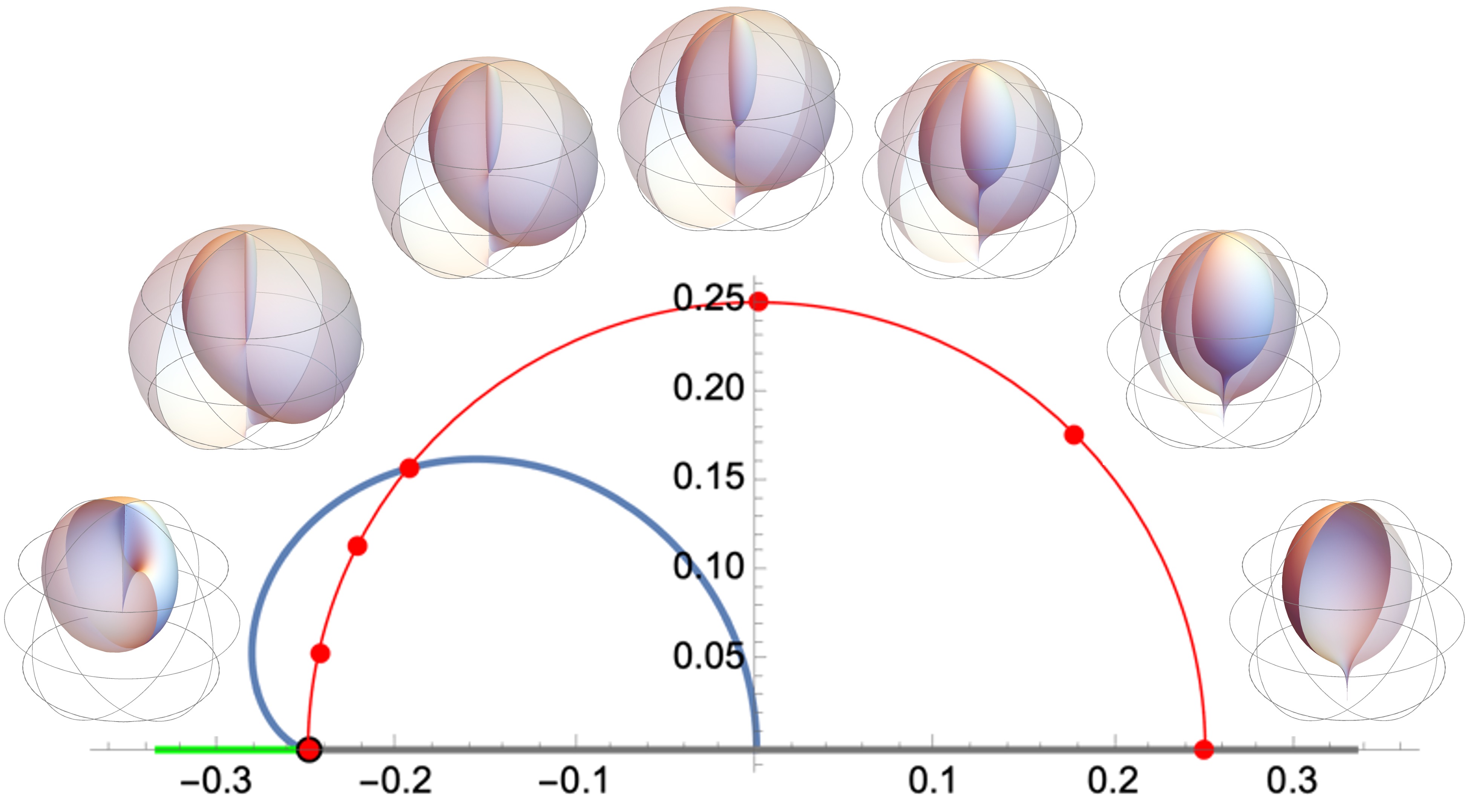}
	\end{center}
	\caption{Associated families of catenoid given by the Weierstrass data $(g,\omega) = (w, \frac{1}{4} \frac{\dif{w}}{w^2})$.}
	\label{Fig:cat2}
\end{figure}

\begin{figure}
	\begin{center}
		\includegraphics[width=0.8\textwidth]{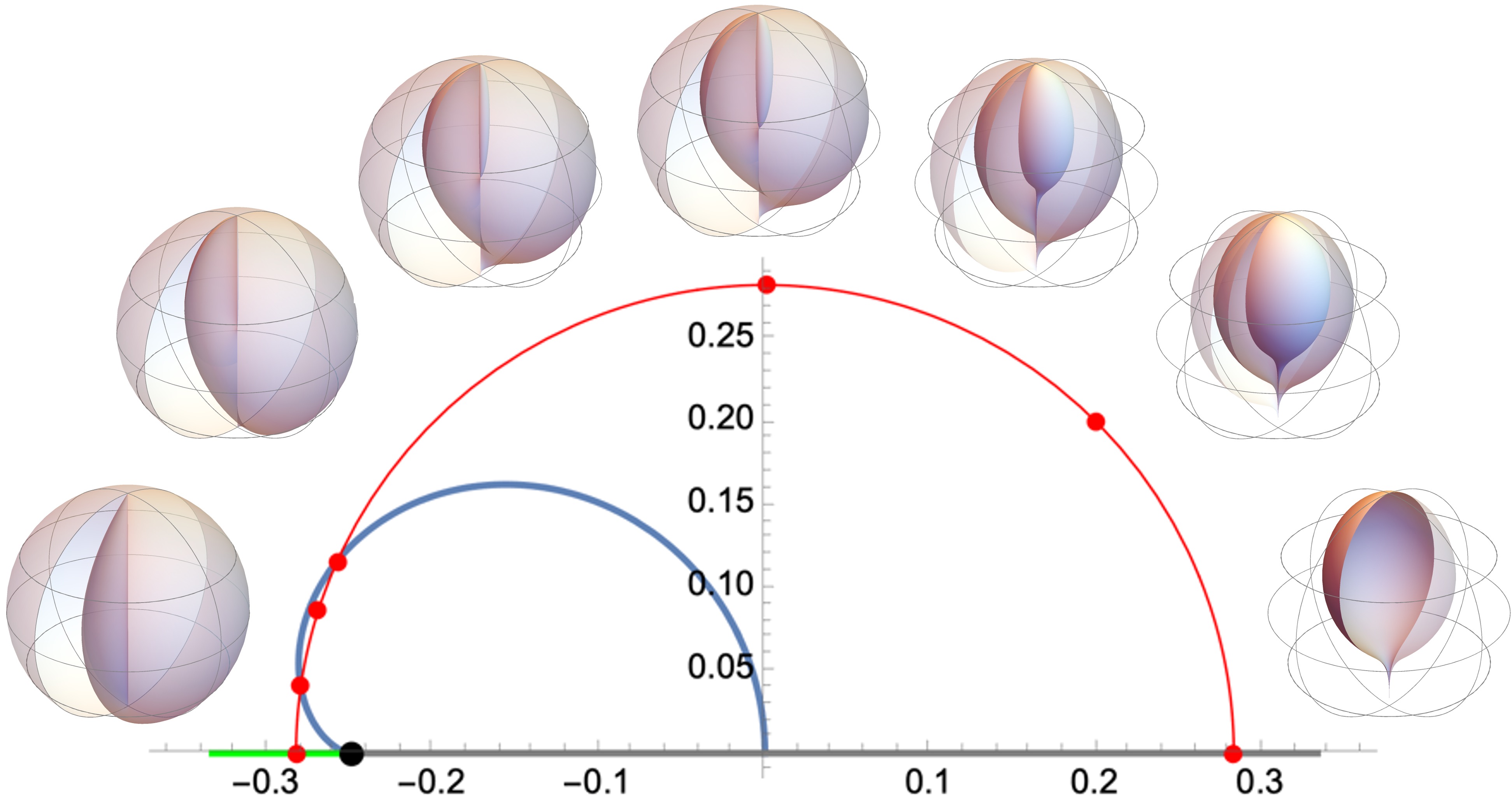}
	\end{center}
	\caption{Associated families of catenoid given by the Weierstrass data $(g,\omega) = (w, \frac{1}{\sqrt{13}} \frac{\dif{w}}{w^2})$.}
	\label{Fig:cat3}
\end{figure}

\begin{figure}
	\begin{center}
		\includegraphics[width=0.8\textwidth]{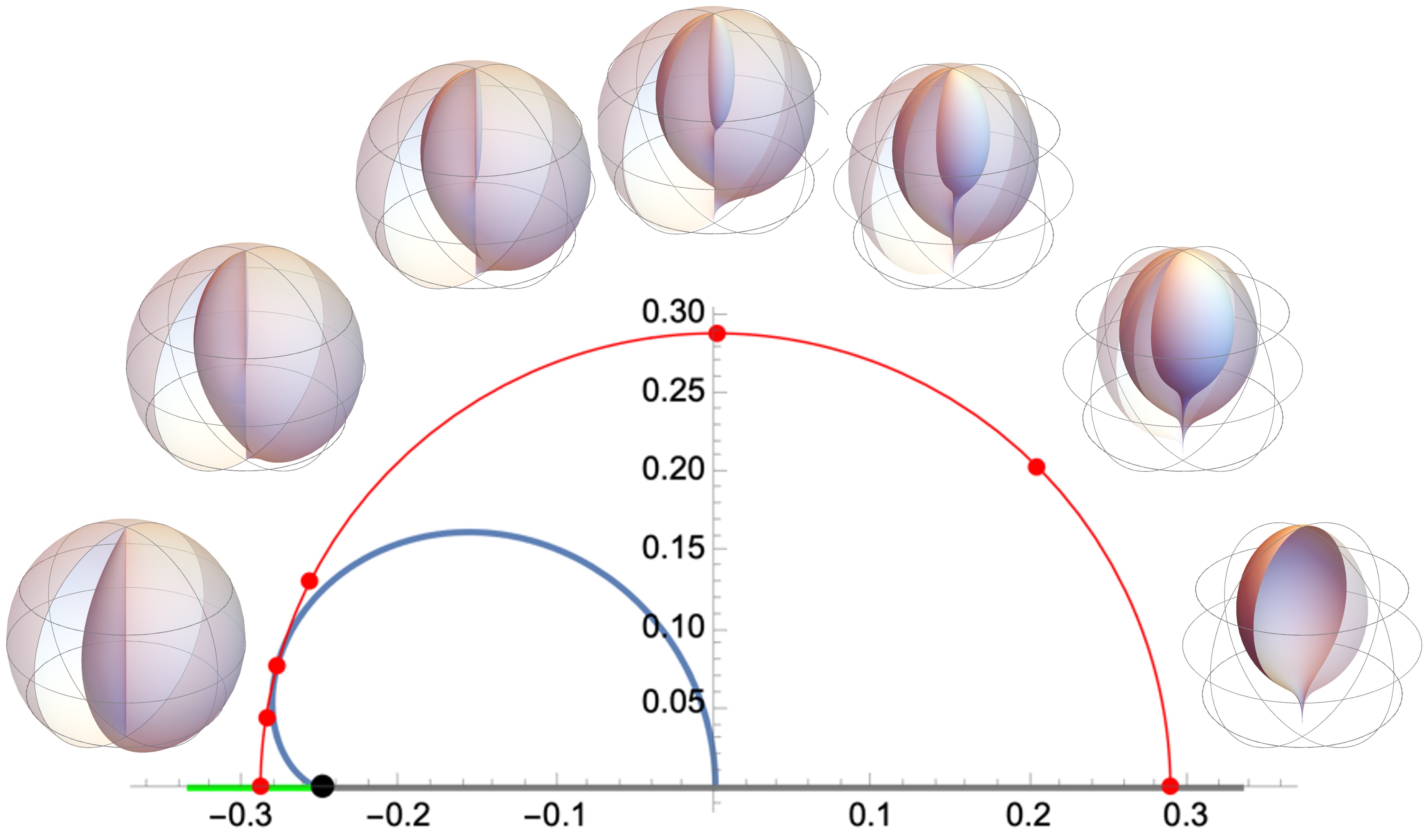}
	\end{center}
	\caption{Associated families of catenoid given by the Weierstrass data $(g,\omega) = (w, \frac{1}{2\sqrt{3}} \frac{\dif{w}}{w^2})$.}
	\label{Fig:cat4}
\end{figure}

\subsection{Associated family of catenoids under Lawson-type correspondence between $\ithree$ and $\qthreep$}\label{sec:lawson}
Lawson-type correspondence between ZMC surfaces in the isotropic $3$-space $\ithree$ and the $3$-dimensional lightcone $\qthreep$ has been established in \cite{pember_WeierstrasstypeRepresentations_2020} as a generalization of the Umehara-Yamada perturbation \cite{umehara_ParametrizationWeierstrassFormulae_1992}.

In this section, we derive a similar type of correspondence between ZMC surfaces in $\ithree$ and in $\qthreep$, and consider the correspondence between the associated family of catenoids in $\ithree$ and $\qthreep$.
The derivation is obtained efficiently by viewing both classes of surfaces as graphs:
Suppose that $X : \mathcal{U} \to \qthreep$ is an immersion represented as a graph of a function $f: \mathcal{U} \to \mathbb{R}$, that is,
	\[
		X(u,v) = e^{f(u,v)} \begin{pmatrix} u^2+v^2 & u +iv \\ u-iv & 1 \end{pmatrix}.
	\]
We have seen in \eqref{Eq:202503160928PM} that $X$ is a ZMC surface if and only if $f$ is a harmonic map.
However, it is known \cite{strubecker_DifferentialgeometrieIsotropenRaumes_1942a} that any ZMC surface in $\ithree$ must be a graph of a harmonic map, giving us the following correspondence between ZMC surfaces in $\ithree$ and $\qthreep$:
\begin{lemma}\label{Lem:202502020719PM}
	The following map
		\begin{equation}\label{Eq:202501110711AM}
			T : \ithree \to \qthreep, \qquad
			T(\rmx,\rmy, \ell)  :=    e^\ell \begin{pmatrix} \rmx^2+\rmy^2 & \rmx+i\rmy \\ \rmx-i\rmy & 1 \end{pmatrix}
		\end{equation}
	sends ZMC surfaces in $\ithree$ to ZMC surfaces in $\qthreep$.
\end{lemma}
Note that
	\[
		T(\rmx,\rmy,\ell) 
		=
		P(\rmx+i\rmy) D_1(e^{-\ell/2}) \begin{pmatrix} 0 \\ 1 \end{pmatrix}
		\left( P(\rmx+i\rmy) D_1(e^{-\ell/2}) \begin{pmatrix} 0 \\ 1 \end{pmatrix} \right)^\star.
	\]

\begin{example}
The image of the plane $ \ell = f(\rmx,\rmy) = a \rmx + b \rmy +c$ in $\ithree$ under $T$ in \eqref{Eq:202501110711AM} is
	\[
		X(u,v) = e^{2(a u + b v + c)} \begin{pmatrix}  u^2+v^2 & u+iv \\ u-iv & 1 \end{pmatrix}, 
	\]
which is the parabolic catenoid.
\end{example}

We shall now examine the relationship between the associated families of catenoids in $\ithree$ and $\qthreep$.
For arbitrary real constants  $\alpha$ and $\beta$, consider
	\[
		\Phi(\tilde{z}) :=  e^{\alpha+i\beta} (\tilde{z}, -i\tilde{z}, \ln{\tilde{z}}) \in \mathbb{C}^3,\qquad \tilde{z}:=e^{u+iv} \in \mathbb{C}.
	\]
With $r:=e^{\alpha}$,
	\[
		\tilde{X}_{\alpha,\beta}(u,v) := \mathrm{Re}(\Phi(z))
		= r \left( e^{u} \cos(v+\beta), e^{u} \sin(v+\beta),  u \cos\beta - v \sin\beta \right)
	\]
are associate families of catenoids under homotheties in $\ithree$.
By Lemma~\ref{Lem:202502020719PM}, 
	\[
		X_{\alpha,\beta}(u,v) := e^{r (u\cos\beta - v\sin\beta)} 
			\begin{pmatrix}  e^{2(u+\alpha)} & e^{u+\alpha} e^{i(v+\beta)}  \\  e^{u+\alpha} e^{-i(v+\beta)}   &  1  \end{pmatrix} 
	\]
is a ZMC surface in $\qthreep$  for each $\alpha, \beta \in \mathbb{R}$.
Then using $c := \frac{1}{2}e^{\alpha+i\beta} \in \mathbb{C}\setminus\{0\}$, we can rewrite $X_{\alpha,\beta} = X_c$ as
	\begin{equation}\begin{aligned} \label{eqn:xab}
		X_c (u,v)
			&= X_{\alpha,\beta}(u,v) \\
			&= e^\alpha e^{r (u\cos\beta - v\sin\beta)}  D_1(e^{\frac{1}{2}(\alpha+i \beta)})
				\begin{pmatrix}
					e^{2u} & e^{u+i v} \\
					e^{u-iv} & 1
				\end{pmatrix}D_1(e^{\frac{1}{2}(\alpha+i \beta)})^\star\\
			&= e^\alpha D_1(e^{\frac{1}{2}(\alpha+i \beta)})\varphi_{c}(z) \varphi_{c}(z)^\star D_1(e^{\frac{1}{2}(\alpha+i \beta)})^\star\\
			&\simeq \varphi_{c}(z) \varphi_{c}(z)^\star
	\end{aligned}\end{equation}
where $\varphi_{c}(z) := e^{c z} \begin{pmatrix} e^z \\ 1 \end{pmatrix}$.

To find the Weierstrass data for $X_c$, we use \eqref{eqn:xab} to assume without loss of generality that
	\[
		X_{c}(z) = \varphi_{c}(z) \varphi_{c}(z)^\star,
	\]
that is, the lift of $X_c$ to $\mathbb{C}^2$ is $\varphi_{c}(z)$.
Then the lift $F_c$ of $X_c$ to $\mathrm{SL}(2,\mathbb{C})$ can be found using \eqref{Eq:202503230200PM} as
	\[
		F_{c}(z) :=
			\begin{pmatrix} e^{\tfrac{z}{2}} & 0 \\ 0 & e^{-\tfrac{z}{2}} \end{pmatrix}
			\begin{pmatrix}
				1 & -\frac{(c+1)^2}{2 c+1} \\
				 1 & -\frac{c^2}{2 c+1}
			\end{pmatrix}
			\begin{pmatrix} e^{cz} & 0 \\ 0 & e^{-cz} \end{pmatrix}
			\begin{pmatrix} e^{\tfrac{z}{2}} & 0 \\ 0 & e^{-\tfrac{z}{2}} \end{pmatrix},
	\]
and we can see that $F_c$ is indeed null-holomorphic.
Direct calculations then show that
	\begin{align*}
		(G,\Omega) &= \left((1 + \tfrac{1}{c})e^z, -c^2 e^{-z} \dif{z}\right),\\
		(g,\omega) &=\left(-\tfrac{c (c+1) }{2 c+1} e^{-(2 c+1) z},  -e^{(2 c +1)z}\dif{z}\right).
	\end{align*}
%
Now we make a coordinate change $w := g$ so that
	\[
		(\tilde{g}_c, \tilde{\omega}_c) = \left(w, -\frac{c(c+1)}{(2c+1)^2}\frac{\dif{w}}{w^2}\right).
	\]
Since the map $\tilde{\delta}(c)$ given by
	\[
		\tilde{\delta}(c) = -\frac{c(c+1)}{(2c+1)^2}
	\]
is a surjection onto $\mathbb{C} \setminus \{ -\frac{1}{4}\}$, we conclude that every surface in the associated families of elliptic or hyperbolic catenoids in $\qthreep$ corresponds to a surface in the associated familiy of catenoids in $\ithree$, under the correspondence in Lemma~\ref{Lem:202502020719PM}.

In particular, if $X_c$ is a ZMC surface constructed using Weierstrass data $(\tilde{g}_c, \tilde{\omega}_c)$, then it is
	\begin{itemize}
		\item an elliptic catenoid when $c \in \mathbb{R}\setminus\{\tfrac{1}{2}\}$ so that $\tilde{\delta}(c) \in (-\frac{1}{4},\infty)$,
		\item a hyperbolic catenoid when $c = -\frac{1}{2} + i \tilde{c}$ for $\tilde{c} \in \mathbb{R} \setminus \{0\}$ so that $\tilde{\delta}(c) \in (-\infty, -\frac{1}{4})$, or
		\item a helicoid when $c \in i \mathbb{R} \setminus \{0\}$.
	\end{itemize}

\subsection{Lightlike Gauss maps of catenoids and helicoids}\label{Subsec:202503230138PM}
Let $X: \mcU \to \qthreep$ be a conformally parametrized ZMC immersion with conformal coordinates $(u,v)$.
Then the lightlike Gauss map $\G : \mcU \to \qthreem$ can be viewed as a surface into $\qthreem$, called the \emph{associated surface} of $X$ in \cite{Liu2}, and the following facts are known:
\begin{fact}
Let $X: \mcU \to \qthreep$ be a ZMC immersion with conformal coordinates $(u,v) \in \mcU$.
If the lightlike Gauss map $\G$ is immersed, then $\G$ is also a conformally parametrized ZMC immersion. Furthermore, the first fundamental form $\mathbf{g}_\G$ of $\G$ is given in terms of the first fundamental form of $X$ via
	\[
		\mathbf{g}_\G=-\mathrm{K}\mathbf{g}_X.
	\]
\end{fact}

Normalizing helicoids $\tilde{H}^a (u,v)$ in \eqref{Eq:202501280355AM} with $b = \frac{1}{2}$ so that
	\[
		\tilde{H}^a (u,v) = e^{-2av}\begin{pmatrix} e^{2u} & e^{u+iv} \\ e^{u-iv} & 1 \end{pmatrix},
	\]
the lightlike Gauss map $\G^a(u,v)$ of $\tilde{H}^{a}(u,v)$ is given by
	\[
		\G^a(u,v) = -2 e^{2av} \begin{pmatrix}
				a^2 + 1 & a(a-i) e^{-u+iv} \\
				a(a+i) e^{-u-iv} & a^2 e^{-2u}
			\end{pmatrix}.
	\]
Setting $e^{\alpha + i \beta} := a - i$ for $\alpha,\beta \in \mathbb{R}$, we then notice
	\begin{align*}
		\G^a(u,v) &= -2 e^{2av} \begin{pmatrix}
				e^{2 \alpha} & a e^{\alpha + i \beta} e^{-u+iv} \\
				ae^{\alpha - i \beta} e^{-u-iv} & a^2 e^{-2u}
			\end{pmatrix}\\
			&= -2a e^\alpha D_2(\tfrac{i}{\sqrt{a}})D_1(e^{-\frac{1}{2}(\alpha + i \beta)}) \tilde{H}^{a}(-u, -v) D_1(e^{-\frac{1}{2}(\alpha + i \beta)})^\star D_2(\tfrac{i}{\sqrt{a}})^\star \\
			&\simeq \tilde{H}^{a}(-u, -v),
	\end{align*}
so that $\G^a(u,v)$ is again a helicoid.
We can similarly check that if a ZMC surface in $\qthreep$ is a catenoid, then its lightlike Gauss map is also a catenoid:
\begin{proposition}
	Let $X: \mcU\to \qthreep$ be a helicoid or (an elliptic, a hyperbolic, a parabolic) catenoid.
	Then the lightlike Gauss map of $X$ is also a helicoid or (an elliptic, a hyperbolic, a parabolic) catenoid in $\qthreem$, respectively.
\end{proposition}


\section{Classification of ruled ZMC surfaces in $\qthreep$}\label{Section5}

In this section, we show that any ruled ZMC surface in $\qthreep$ must be either a helicoid \eqref{Eq:202501280355AM} or a parabolic catenoid \eqref{Eq:202501280440} up to isometry and homothety.
As noted in Lemma~\ref{lemma:helicoid}, the surface obtained by applying screw motions to a geodesic result in a ZMC surface, which we call helicoid; however, helices and geodesics do not meet orthogonally (even though they meet at a constant angle).
Therefore, the standard techniques of classifying ruled ZMC surface do not work in $\qthreep$.

To overcome this obstacle, we will view ruled surfaces as application of a curve of isometries to a geodesic, and find the condition on the isometries for the resulting surface to have ZMC.

Thus, without loss of generality, we assume that any spacelike ruled surface in $\qthreep$ is parameterized as
	\[
		X(s,t) = F(s) \delta(t) F(s)^\star
	\]
for a geodesic
$
\delta(t) := \begin{pmatrix} t^2 & t \\ t & 1 \end{pmatrix}
$ 
and some curve $F$ in $\sltc$ with $F(0)=I_2$.
Let 
	\begin{align*}
		X^F_*(s,t) &:= F(s)^{-1}X_*(s,t) (F(s)^{-1})^\star,  \\
		\Omega(s) &:= F(s)^{-1} F(s)' = \begin{pmatrix} \alpha(s) & \beta(s) \\ \gamma(s) & -\alpha(s) \end{pmatrix}.  
	\end{align*}
Then, direct calculations show that
	\begin{gather*}
		X^F(s,t)  = \delta(t), \quad
		X^F_s(s,t)  = \Omega(s) \delta(t) + \delta(t) \Omega(s)^\star, \quad
		X^F_t(s,t)  = \delta(t)', \\ 
		X^F_{ss}(s,t)  = F(s)^{-1} F(s)'' \delta(t) + 2 \Omega(s) \delta(t) \Omega(s)^\star + \delta(t) (F(s)^{-1} F(s)'')^\star, \\
		X^F_{st}(s,t)  = \Omega(s)\begin{pmatrix} 2t & 1 \\ 1 & 0 \end{pmatrix} + \begin{pmatrix} 2t & 1 \\ 1 & 0 \end{pmatrix} \Omega(s)^\star, \qquad
		X^F_{tt}(s,t) = \begin{pmatrix} 2 & 0 \\ 0 & 0 \end{pmatrix}. 
	\end{gather*}
Here we note that
	\[
		F(s)^{-1} F(s)'' = (F^{-1} F')' + (F^{-1} F')^2= \Omega(s)' - \det{\Omega(s)} I_2
	\]
where $I_2$ is the $2\times2$ identity matrix.

Now let $\G(s,t)$ be the lightlike Gauss map of $X(s,t)$.
Then
	\[
		\G^F(s,t) := F(s)^{-1}\G(s,t) (F(s)^{-1})^\star
	\]
satisfies
	\[
		\langle \G^F, \G^F \rangle = \langle \G^F, X^F_s \rangle = \langle \G^F, X^F_t \rangle = \langle \G^F, X^F \rangle -1 = 0
	\]
To proceed further, let
	\[
		a_0 := \begin{pmatrix} -2 & 0 \\ 0 & 0 \end{pmatrix}, \quad
		a_1 := \begin{pmatrix} 0 & 1 \\ 1 & 0 \end{pmatrix}, \quad
		a_2 := \begin{pmatrix} 0 & i \\ -i & 0 \end{pmatrix}, \quad
		a_3 := \begin{pmatrix} 0 & 0 \\ 0 & 1 \end{pmatrix}, 
	\]
which form an asymptotic basis.

We realize that $\delta(t)$ is obtained by rotating $a_3$ by the parabolic rotation $F(t) := \begin{pmatrix} 1 & t \\ 0 & 1 \end{pmatrix}$.
That is, $\delta(t) =  F(t) a_3 F(t)^\star$.  We let
	\[
		\f_{\alpha}(t) := F(t) a_\iota F(t)^\star, \quad \iota=0,1,2,3.
	\]
Then for any $t$, $\{ \f_0, \f_1, \f_2, \f_3 \}$ is an asymptotic basis and we see that%
	\begin{align*}
		& X^F(s,t) = \f_3(t), \quad  X^F_t(s,t) = \f_1(t), \\
		&\Omega(s) = \frac{1}{2}(\alpha(t^2-1) + (\beta+\gamma)t) \f_0 + \frac{1}{2}(2\alpha t + \beta + \gamma) \f_1 - \frac{i}{2} (\beta-\gamma) \f_2 - \alpha \f_3, \\
		& X^F_s(s,t) =  (2t\alpha_1(s) + \beta_1(s) - t^2 \gamma_1(s)) \f_1(t)  \\
		&\qquad\qquad\qquad + (2t \alpha_2(s) + \beta_2(s) - t^2 \gamma_2(s)) \f_2(t) -2 (\alpha_1(s) - t \gamma_1(s)) \f_3(t),
\end{align*}
where $*_1 := \operatorname{Re}(*)$, $*_2 := \operatorname{Im}(*)$  for $*=\alpha, \beta, \gamma$.

Then 
	\begin{multline*}
		\qquad
		\G^F(s,t) = \f_0(t) + \frac{2 (\alpha_1(s) - t \gamma_1(s))}{2t \alpha_2(s) + \beta_2(s) - t^2 \gamma_2(s)} \f_2(t)  \\
		-\frac{2 (\alpha_1(s) - t \gamma_1(s))^2}{2t \alpha_2(s) + \beta_2(s) - t^2 \gamma_2(s)} \f_3(t)
		\qquad\qquad
	\end{multline*}
and
	\begin{align*}
		E &:= \langle X_s, X_s \rangle = \langle X^F_s, X^F_s \rangle, &
			L &:= \langle \G, X_{ss} \rangle = \langle \G^F, X^F_{ss} \rangle, \\
		F &:= \langle X_s, X_t \rangle = \langle X^F_s, X^F_t \rangle, &
			M &:= \langle \G, X_{st} \rangle = \langle \G^F, X^F_{st} \rangle, \\
		G &:= \langle X_t, X_t \rangle = \langle X^F_t, X^F_t \rangle, &
			N &:= \langle \G, X_{tt} \rangle = \langle \G^F, X^F_{tt} \rangle.
	\end{align*}
and the numerator of the mean curvature $\mathrm{H} = \frac{1}{2} \frac{EN-2FM+GL}{EG-F^2}$ of $X$ is a fourth-order polynomial 
	\[
		c_0(s) + c_1(s) t + c_2(s) t^2 + c_3(s) t^3 + c_4(s) t^4
	\]
of $t$ where
	\begin{align}
		c_4(s) &:=  -2 \gamma_2(s)    \left(\gamma_1(s)^2 + \gamma_2(s)^2 \right),  \nonumber \\
		c_3(s) &:= 8 \alpha_2(s)  \left( \gamma_1(s)^2+\gamma_2(s)^2\right)-2 \gamma_2(s) \gamma_1'(s)+2\gamma_1(s) \gamma_2'(s). \label{Eq:202404260724AM}
	\end{align}
Thus $X$ has ZMC if and only if 
	\[
c_0(s) = c_1(s) = c_2(s) = c_3(s) = c_4(s) = 0 \qquad\text{ for all } s.
	\]
From $c_4(s)=0$, 
we see that $\gamma_2(s)=0$ for all $s$.
Now to use \eqref{Eq:202404260724AM}, we distinguish two cases: $\alpha_2(s_0) \not=0$ for some $s_0$ or  $\alpha_2(s)=0$ for all $s$.

\textbf{Case 1.}
$\alpha_2(s_0) \not=0$ for some $s_0$.
By restricting our attention to the interval around $s_0$, we may assume without loss of generality that  $\alpha_2(s) \not=0$ for all $s$.
From the vanishing of $c_3(s)$ in \eqref{Eq:202404260724AM}, we conclude that $\gamma_1(s)=0$ for all $s$.
Then $\Omega(s)$ is an upper triangular matrix, and an interesting analysis can be carried out.
In this case, the coefficients of the numerator of the ZMC surface equation become the following:
	\begin{align*}
		c_0(s) &=2 \beta _2(s) \alpha _1(s)'+\alpha _1(s) \left(4 \alpha_2(s) \beta _1(s)-2 \beta _2(s)'\right)-4 \alpha _1(s)^2 \beta _2(s), \\
		c_1(s) &= 4 \alpha _2(s) \alpha _1(s)'-4 \alpha _1(s) \alpha _2(s)', \qquad
		c_2(s) = c_3(s) = c_4(s) =0.
	\end{align*}
From the vanishing of $c_1(s)$, we conclude that
	\begin{equation}\label{Eq:202404260746AM}
		\alpha_1(s) = d \, \alpha_2(s) 
	\end{equation} 
for some real constant $d$. We distinguish two cases.

\textbf{Case 1-1.} Suppose that $d \not= 0$ so that
	\begin{equation}\label{Eq:202405040409AM}
		\alpha_1(s) \not= 0  \qquad\text{for all s}. 
	\end{equation}
Let $c=1/d$.
Then,
	\[
		\frac{c_0(s)}{\alpha_1(s)^2} = - 4 \beta_2(s) - 2 \left(\frac{\beta_2(s)}{\alpha_1(s)} \right)' + 4 c \beta_1(s),
	\]
and $c_0(s)=0$ is equivalent to
	\begin{equation} \label{Eq:202404260747AM}
		\beta_1(s) = \frac{1}{c} \beta_2(s) + \frac{1}{2c} \left(\frac{\beta_2(s)}{\alpha_1(s)}\right)'.
	\end{equation}
Thus, $X$ has ZMC if and only if \eqref{Eq:202404260746AM} and \eqref{Eq:202404260747AM} hold while $\alpha_1$ and $\beta_2$ are arbitrary smooth functions.

By defining $f$ and $g$ by
	\begin{equation}\label{Eq:202502110809AM}
		\alpha_1(s) := g(s), \qquad   
		\beta_2(s) := f(s) g(s),
	\end{equation}
we have
	\[
		\alpha_2(s) := c g(s), \qquad
		\beta_1(s) := \frac{1}{c} f(s) g(s) + \frac{1}{2c} f(s)',
	\]
so that
	\[
		\Omega(s) = \Omega_1(s) + \Omega_2(s)
	\]
where
	\[
		\Omega_1(s) = \begin{pmatrix} (1+ic) g(s) & 2 (1+ic) g(s)\frac{f(s)}{2c} \\ 0 & -(1+ic) g(s) \end{pmatrix}, \qquad
		\Omega_2(s) = \begin{pmatrix} 0 & \left( \frac{f(s)}{2c} \right)' \\ 0 & 0 \end{pmatrix}.
	\]
	
With $G(s) := \int g(s) ds$,
	\begin{equation} \label{Eq:202404260816AM}
		F_1(s) := \begin{pmatrix} e^{(1+ic) G(s)} & 0 \\ 0 & e^{-(1+ic) G(s)} \end{pmatrix}, \qquad
		F_2(s) := \begin{pmatrix} 1 & \frac{f(s)}{2c}  \\ 0 & 1 \end{pmatrix},
	\end{equation}
and the matrix-valued function 
	\[
		F(s) := F_1(s) \, F_2(s)
	\]
is the (unique) solution to $F(s)^{-1} F(s)' = \Omega(s)$ with $F(0)=I_2$. 
Therefore 
	\[
		X(s,t) = F_1(s) F_2(s) \delta(t) F_2(s)^\star F_1(s)^\star 
	\]
with $F_1, F_2$ as in \eqref{Eq:202404260816AM} for some smooth real-valued functions $f, G$.
Note that for fixed $s$, the image of $t \mapsto F_2(s) \delta(t) F_2(s)^\star$ is the same as the image of $t \mapsto \delta(t)$. Hence, $X(s,t) = F_1(s)  \delta(t) F_1(s)^\star$.

Finally we notice from \eqref{Eq:202405040409AM} and \eqref{Eq:202502110809AM} that $g(s) \not=0$, hence we conclude that $G$ is a strictly monotone function, hence it has an inverse function $s = h(G)$ and we can take $G$ as a new variable.
We abuse notation by calling it $s$ again, and conclude that  
	\[
		X(s,t) = 
			\begin{pmatrix} e^{(1+ic) s} & 0 \\ 0 & e^{-(1+ic) s} \end{pmatrix}
			\begin{pmatrix} t^2 & t \\ t & 1 \end{pmatrix}
			\begin{pmatrix} e^{(1+ic)s} & 0 \\ 0 & e^{-(1+ic) s} \end{pmatrix}^\star.
	\]
Then for any nonzero real number $a$, we may set $s=a s'$ and $b=ca$, and by calling $s'$ as $s$ again, and conclude that
	\begin{equation}\label{Eq:202405040406AM}
		X(s,t) = 
			\begin{pmatrix} e^{(a+ib) s} & 0 \\ 0 & e^{-(a+ib) s} \end{pmatrix}
			\begin{pmatrix} t^2 & t \\ t & 1 \end{pmatrix}
			\begin{pmatrix} e^{(a+ib) s} & 0 \\ 0 & e^{-(a+ib) s} \end{pmatrix}^\star.
	\end{equation}
That is, any generic ruled ZMC surface in $\qthreep$ is congruent to one of the above. 

It is already shown that the surface in \eqref{Eq:202405040406AM} is a ruled ZMC surface in 
$\qthreep$.

\textbf{Case 1-2.} Suppose that $d=0$.
Then $\alpha_1(s)=0$ for all $s$, from which it follows that $c_0(s)$  and $c_1(s)$ are also 0 for all $s$.
In this case
	\[
		\Omega(s) = \begin{pmatrix}   i \alpha_2(s)  &  \beta_1(s) + i \beta(2)  \\ 0  & -  i \alpha_2(s) \end{pmatrix}.
	\]
We can easily solve $F(s)^{-1} F(s)'=\Omega(s)$ with $F(0)=I_2$ to obtain
	\begin{equation}\label{202405040544AM}
		F(s) = 
			\begin{pmatrix} 1 & B_1(s) + i B_2(s) \\ 0 & 1 \end{pmatrix}
			\begin{pmatrix} e^{iA(s)} & 0 \\ 0 & e^{-iA(s)}  \end{pmatrix},
	\end{equation}
where $A(s) := \int_0^s \alpha_2(\tilde{s}) d\tilde{s}$, $B_1(s) + i B_2(s) := \int_0^s (\beta_1(\tilde{s}) + i \beta_2(\tilde{s}) ) e^{2iA(\tilde{s})} d\tilde{s}$.
So we conclude that $F(s)$ is of the form in \eqref{202405040544AM} for some real-valued functions $A,B_1,B_2$.

\textbf{Case 2.} $\alpha_2(s)=0$ for all $s$.
Then $c_2(s)=-6 \beta_2(s) \gamma_1(s)^2$.
We again have two cases:

\textbf{Case 2-1.} Suppose that $\beta_2(s_1)\not=0$ for some $s_1$.
By continuity, we may assume without loss of generality that $\beta_2(s)\not=0$ for all $s$, so that $\gamma_1(s)=0$ for all $s$. 
Then 
	\begin{align*}
		& c_0(s) = \alpha_1(s)' \beta_2(s) - \alpha_1(s) \beta_2(s)' - 2 \alpha_1(s)^2 \beta_2(s), \\
		& c_1(s) = c_2(s) = c_3(s) = c_4(s) =0.
	\end{align*}
Hence the mean curvature is identically zero if and only if 
	\[
		\alpha_1(s)' \beta_2(s) - \alpha_1(s) \beta_2(s)' - 2 \alpha_1(s)^2 \beta_2(s) =0,
	\]
or equivalently
	\[
		\beta_2(s) = c_1 \alpha_1(s) e^{-2  \int_0^s \alpha_1(\tilde{s}) \, d\tilde{s}} = c_2 A(s)' e^{-A(s)}.
	\]
In this case,
	\begin{equation}\label{Eq:202405040711PM}
		\Omega(s) = \begin{pmatrix} \alpha_1(s) & \beta_1(s) + i \beta_2(s) \\ 0 & - \alpha_1(s)\end{pmatrix},
	\end{equation}
and $F^{-1} F' = \Omega$ with $F(0)=I_2$ yields
	\begin{equation} \label{Eq:202405040407AM}
		F(s) = \begin{pmatrix} 1 & B_1(s) + i B_2(s) \\ 0 & 1 \end{pmatrix}
			\begin{pmatrix} e^{A(s)/2} & 0 \\ 0 & e^{-A(s)/2} \end{pmatrix}
	\end{equation}
where $A(s) := 2 \int_0^s \alpha_1(\tilde{s}) d\tilde{s}$,  $B_1(s) := \int_0^s \beta_1(\tilde{s}) e^{A(\tilde{s})} d\tilde{s}$, and  $B_2(s) := c_2 A(s)$.

Conversely, if $F$ is defined by \eqref{Eq:202405040407AM} with arbitrary $A(s)$ and $B_1(s)$ with $B_2(s) := c_2 A(s)$ where $c_2$ is an arbitrary real number, then $\Omega(s) := F(s)^{-1} F(s)'$ satisfies \eqref{Eq:202405040711PM} with
$ \alpha_1(s) = \tfrac{1}{2} A(s)', \beta_1(s) = B_1(s)' e^{-A(s)}, \beta_2(s) = B_2(s)' e^{-A(s)}$.

\textbf{Case 2-2.} Suppose that $\beta_2(s)=0$ for all $s$.
Then coefficients $c_0, c_1, c_2, c_3, c_4$ are all zero.
In this case, $\Omega(s)$ is real and trace-free, i.e. $\Omega(s) \in \lasltr$.
Then $F$ such that $F^{-1} F' = \Omega$ is real-valued, i.e. $F(s) \in \sltr$.
This means that $X = F(s)\delta(t)F(s)^\star$ is real-valued, which in turn means that the $\rmxy$-component of $X(s,t)$ is zero, hence the image of $X$ lies in 
	\[
		\qthreep \cap \{ \rmxx=0\} = \{ (\rmxt,\rmxx,\rmxy,\rmxz) : \rmxt^2 + \rmxx^2 +\rmxz^2 =0, \ \rmxy=0 \}.
	\]
Since it is not spacelike, we exclude this case.  

In summary, we have the following:
\begin{proposition}
	A surface given by $X(s,t) := F(s) \delta(t) F(s)^\star$ where  $\delta(t) := \begin{pmatrix} t^2 & t \\ t & 1 \end{pmatrix}$ and
		\begin{align}
			F(s) &:= \begin{pmatrix} e^{a s} & 0 \\ 0 & e^{-a s} \end{pmatrix}	 \begin{pmatrix} e^{ib s} & 0 \\ 0 & e^{-ib s} \end{pmatrix},	
			\qquad\text{or} 		\nonumber \\
			F(s) &:= \begin{pmatrix} 1 & B_1(s) + i B_2(s) \\ 0 & 1 \end{pmatrix} \begin{pmatrix} e^{iA(s)} & 0 \\ 0 & e^{-iA(s)} \end{pmatrix},
 			\qquad\text{or} 		 	\label{Eq:202501040822AM} \\
			F(s) &:= \begin{pmatrix} 1 & B_1(s) + i c_2 A(s) \\ 0 & 1 \end{pmatrix} \begin{pmatrix} e^{A(s)/2} & 0 \\ 0 & e^{-A(s)/2} \end{pmatrix} \label{Eq:202501041234PM} 
		\end{align}	
	for arbitrary real numbers $a,b,c$ and  arbitrary real-valued functions $A, B_1, B_2$, is a ruled ZMC surface.

	Conversely, any ruled ZMC surface in $\qthreep$ is congruent to one of the above.
\end{proposition}

Now we analyse the shapes of the ruled ZMC surfaces given by \eqref{Eq:202501040822AM} and \eqref{Eq:202501041234PM}. 
First of all, we see that for $F$ in \eqref{Eq:202501040822AM}, 
	\[
		X(s,t) = \begin{pmatrix} * & ** \\ *** & 1 \end{pmatrix},
	\]
that is, the surface is the intersection of $\qthreep$ and the hyperplane $\rmxt - \rmxz =1$ in $\lfour$.
So it is simply the horosphere
	\[
		\{ (\rmxt,\rmxx,\rmxy,\rmxz) \in \qthreep :  \rmxt + \rmxz = \rmxx^2+\rmxy^2\}.
	\]

Now let $X$ be given by $F$ in \eqref{Eq:202501041234PM}.
Direct calculations show that
	\[
		\rmxy(s) = c_2 s e^{-2s}, \qquad
		\rmxt(s) - \rmxz(s) = e^{-2s}.
	\]
Then 
	\begin{equation}\label{Eq:202501040834AM}
		\rmxy = c_3 (\rmxt - \rmxz) \ln(\rmxt - \rmxz), \qquad c_3 = - c_2/2 \in \mathbb{R} \setminus\{0\},
	\end{equation}
that is, the ruled ZMC surface in \eqref{Eq:202501041234PM} is the intersection of $\qthreep$ and the surface given by \eqref{Eq:202501040834AM}.

We show that this is a ruled surface.
If we let
	\[
		\rmxt - \rmxz = e^s,
	\]
then
	\[
		X = \vec{a}(s) \frac{\rmxx^2}{2} + \vec{b}(s) \rmxx + \vec{c}(s)
	\]
where
	\[
		\vec{a}(s) := \begin{pmatrix} 2 e^{-s} & 0 \\ 0 & 0 \end{pmatrix}, \quad
		\vec{b}(s) := \begin{pmatrix} 0 & 1 \\ 1 & 0 \end{pmatrix}, \quad
		\vec{c}(s) := e^s \begin{pmatrix} (c_3 s)^2 & i c_3 s \\ - i c_3 s & 1 \end{pmatrix}.
	\]
If we let
	\[
		\rmxx = u, \quad 
		c_3 s = v, 
	\]
then
	\begin{equation}\label{Eq:202501040851AM}
		X(u,v) = \begin{pmatrix} 2 e^{-cv} & 0 \\ 0 & 0 \end{pmatrix} \frac{u^2}{2} + \begin{pmatrix} 0 & 1 \\ 1 & 0 \end{pmatrix} u  + e^{cv} \begin{pmatrix} v^2 & iv \\ -iv & 1 \end{pmatrix}
	\end{equation}
Note that if $c=0$ then
\[
	X(u,v) = \begin{pmatrix} u^2 + v^2 & u+iv \\ u-iv & 1 \end{pmatrix},
\]
which is the surface given in \eqref{Eq:202501040822AM}.

To see how \eqref{Eq:202501040851AM} is obtained from the geodesic $X(u,0)$, we see that
	\begin{align*}
		X(u,v) 
			&= \left( \Phi_1(v) \Phi_2(v) \begin{pmatrix} u \\ 1 \end{pmatrix} \right) \left( \Phi_1(v) \Phi_2(v) \begin{pmatrix} u \\ 1 \end{pmatrix} \right) ^\star \\
			&= \left( \Phi_1(v) \Phi_2(v) \right) \begin{pmatrix} u^2 & u \\ u & 1 \end{pmatrix} \left( \Phi_1(v) \Phi_2(v) \right)^\star
	\end{align*}
where
	\[
		\Phi_1(v) = \begin{pmatrix} 1 & i v \\ 0 & 1 \end{pmatrix},	\quad
		\Phi_2(v) = \begin{pmatrix} e^{-cv/2} & 0 \\ 0 & e^{cv/2} \end{pmatrix},
	\]
that is, we apply to the geodesic $\Phi(u,0)$ the screw motion which is a composition of the hyperbolic rotation $\Phi_2(v)$ and the parabolic rotation $\Phi_1(v)$.

\begin{theorem}\label{Thm:202503130757AM}
	The following maps
		\[
			H^{a,b}(u,v) := 
				\begin{pmatrix} e^{(a + ib)v} & 0 \\ 0 & e^{-(a + ib)v} \end{pmatrix}	 
				\begin{pmatrix} u^2 & u \\ u & 1 \end{pmatrix} 
				\begin{pmatrix} e^{(a + ib)v} & 0 \\ 0 & e^{-(a + ib)v} \end{pmatrix}^\star		
		\]
	and
		\[
			\tilde{C}^P_c(u,v) :=
				\begin{pmatrix} 1 & iv \\ 0 & 1 \end{pmatrix}	 \begin{pmatrix} e^{ -\frac{cv}{2}} & 0 \\ 0 & e^{\frac{cv}{2}} \end{pmatrix}
				\begin{pmatrix} u^2 & u \\ u & 1 \end{pmatrix} 
				\begin{pmatrix} e^{ -\frac{cv}{2}} & 0 \\ 0 & e^{\frac{cv}{2}} \end{pmatrix}^\star
				\begin{pmatrix} 1 & iv \\ 0 & 1 \end{pmatrix}^\star
		\]
	are ruled ZMC surfaces in $\qthreep$, where $a, b, c$ are real constants with $b \not= 0$.
	$H^{a,b}$ is the helicoid and $C^P_c$ is the parabolic catenoid.
	When $a=0$, $c=0$, they are standard horosphere.

	Conversely, any spacelike ruled ZMC surface in $\qthreep$ must be one of the above up to the homothety and isometries of $\qthreep$.
\end{theorem}
\begin{proof}
	The first claim follows from Subsections~\ref{SubSec:202503231257PM} and \ref{SubSec:202503231258PM}. The second claim follows from the contents of this section.
\end{proof}
\bibliographystyle{abbrv}



\end{document}